\documentclass[a4paper, draft,reqno]{amsart}
\usepackage{amssymb,verbatim,dsfont}
\usepackage{xypic}
\xyoption{all}
\usepackage{verbatim}
\usepackage{enumitem}

\newtheorem{theorem}{Theorem}[section]
\newtheorem{lemma}[theorem]{Lemma}
\newtheorem{proposition}[theorem]{Proposition}
\newtheorem{corollary}[theorem]{Corollary}

\theoremstyle{definition}
\newtheorem{definition}[theorem]{Definition}

\newtheorem{example}[theorem]{Example}

\theoremstyle{remark}
\newtheorem{remark}[theorem]{Remark}

\DeclareMathOperator{\End}{End}

\DeclareMathOperator{\ide}{id}

\DeclareMathOperator{\M}{M}

\DeclareMathOperator{\Z}{Z}

\newcommand{\ot}{\otimes}

\def\xcirc{\objectmargin{0.1pc}\def\objectstyle{\sssize}\diagram
\squarify<1pt>{}\circled\enddiagram}

\begin{document}

\title{Non commutative truncated polynomial extensions}

\author{Jorge A. Guccione}
\address{Departamento de Matem\'atica\\ Facultad de Ciencias Exactas y Naturales-UBA, Pabell\'on~1-Ciudad Universitaria\\ Intendente Guiraldes 2160 (C1428EGA) Buenos Aires, Argentina.}
\address{Instituto de Investigaciones Matem\'aticas ``Luis A. Santal\'o"\\ Facultad de Ciencias Exactas y Naturales-UBA, Pabell\'on~1-Ciudad Universitaria\\ Intendente Guiraldes 2160 (C1428EGA) Buenos Aires, Argentina.}
\email{vander@dm.uba.ar}

\author{Juan J. Guccione}
\address{Departamento de Matem\'atica\\ Facultad de Ciencias Exactas y Naturales-UBA\\ Pabell\'on~1-Ciudad Universitaria\\ Intendente Guiraldes 2160 (C1428EGA) Buenos Aires, Argentina.}
\address{Instituto Argentino de Matem\'atica-CONICET\\ Savedra 15 3er piso\\ (C1083ACA) Buenos Aires, Argentina.}
\email{jjgucci@dm.uba.ar}

\thanks{Jorge A. Guccione and Juan J. Guccione research were supported by UBACYT X095, PIP 112-200801-00900 (CONICET) and PICT 2006 00836 (FONCYT)}

\author{Christian Valqui}
\address{Pontificia Universidad Cat\'olica del Per\'u - Instituto de Matem\'atica y Ciencias Afi\-nes, Secci\'on Matem\'aticas, PUCP, Av. Universitaria 1801, San Miguel, Lima 32, Per\'u.}

\address{Instituto de Matem\'atica y Ciencias Afines (IMCA) Calle Los Bi\'ologos 245. Urb. San C\'esar. La Molina, Lima 12, Per\'u.}
\email{cvalqui@pucp.edu.pe}

\thanks{Christian Valqui research was supported by PUCP-DGI-2010-0025, Lucet 90-DAI-L005, SFB 478 U. M¨unster, Konrad Adenauer Stiftung.}

\thanks{The second author thanks the appointment as a visiting professor ``C\'atedra Jos\'e Tola Pasquel'' and the hospitality during his stay at the PUCP}

\subjclass[2010]{Primary 16S10; Secondary 16S80}
\keywords{Twisting maps, Polynomial rings}

\begin{abstract} We introduce the notion of non commutative truncated polynomial extension of an algebra $A$. We study two families of these extensions. For the first one we obtain a complete classification and for the second one, which we call upper triangular, we find that the obstructions to inductively construct them, lie in the Hochschild homology of $A$, with coefficients in a suitable $A$-bimodule.
\end{abstract}

\date{}

\dedicatory{}


\maketitle

\section*{Introduction}
Let $k$ be a commutative ring and let $A$, $C$ be unitary $k$-algebras. By definition, a  twisted tensor product of $A$ with $C$ over $k$, is an algebra structure defined on $A\ot_k C$, with unit $1\ot 1$, such that the canonical maps $i_A\colon A\to A\ot_k C$ and $i_C\colon C\to A\ot_k C$ are algebra maps satisfying $a\ot c = i_A(a)i_C(c)$. This structure was introduced independently in \cite{Ma1} and \cite{Tam}, and it has been formerly studied by many people with different motivations (In addition to the previous references see also \cite{B-M1}, \cite{B-M2}, \cite{Ca}, \cite{C-S-V}, \cite{C-I-M-Z}, \cite{G-G}, \cite{Ma2}, \cite{J-L-P-V}, \cite{VD-VK}).  A number of examples of classical and recently defined constructions in ring theory fits into this construction. For instance, Ore extensions, skew group algebras, smash products, etcetera (for the definitions and properties of these structures we refer to \cite{Mo} and \cite{Ka}). On the other hand, it has been applied to braided geometry and it arises as a natural representative for the product of noncommutative spaces, this being based on the existing duality between the categories of algebraic affine spaces and commutative algebras, under which the cartesian product of spaces corresponds to the tensor product of algebras. And last, but not least, twisted tensor products arise as a tool for building algebras starting with simpler ones.

Given algebras $A$ and $C$, a basic problem is to determine all the twisted tensor products of $A$ with $C$. To our knowledge, the first paper in which this problem was attacked in a systematic way was \cite{C}, in which C. Cibils studied and completely solved the case $C:= k\times k$. Subsequently, in \cite{J-L-N-S}, the methods developed in \cite{C} were extended to cover the case $C:= k\times \cdots \times  k$ ($n$-times). Meanwhile, in \cite{G-G-V}, some partial results were obtained in the cases $C:= k[x]$ and \mbox{$C:=k[[x]]$}.

In this paper we consider this problem when $C$ is a truncated polynomial algebra $k[y]/\langle y^n \rangle$. We call these twisted tensor products {\em non commutative truncated polynomial extensions of $A$}, because they have underlying module $A[y]/\langle y^n\rangle$ and include $A$ and $k[y]/\langle y^n \rangle$ as subalgebras.

It is well known that there is a canonical bijection between the twisted tensor products of $A$ with $C$ and the so called twisting maps $s\colon C\ot_k A \to A\ot_k C$. So each twisting map $s$ is associated with a twisted tensor product of $A$ with $C$ over $k$, which will be denoted by $A\ot_s C$.

It is evident that each $k$-linear map $s\colon k[y]/\langle y^n \rangle \ot_k A \to A \ot_k k[y]/\langle y^n \rangle$ determines and it is determined by $k$-linear maps $\gamma^r_j\colon A\to A$ ($0\le j,r < n$) such that
\begin{equation}
s(y^r\otimes a) = \sum_{j=0}^{n-1}\gamma_j^r(a)\otimes y^j.\label{eq1intro}
\end{equation}
The map $s$ so defined is a twisting map if the maps $\gamma_j^r$ satisfy suitable conditions. In particular, we will see that $B:=\ker \gamma^1_0$ should be a subalgebra of $A$, and $\gamma^1_0$ a nilpotent right $B$-linear map.

The main results of this paper are the following: Theorem~\ref{clasificacion}, which determines all the twisting maps such that
\begin{itemize}

\smallskip

\item[-] $B$ is a subalgebra of the center of $A$,

\smallskip

\item[-] $s(B\ot_k A)\subseteq A\ot_k B$,

\smallskip

\item[-] there exist $h\ge 2$ and $x\in A$ such that $\gamma^h_0=0$ and $\gamma^{h-1}_0(x)$ is invertible,

\end{itemize}
and Theorem~\ref{obstruccion en la cohomologia}, which establish that the obstruction to ``extend'' a twisting map
$$
s_n \colon \frac{k[y]}{\langle y^n \rangle} \ot_k A \to A \ot_k \frac{k[y]}{\langle y^n \rangle}
$$
with $\gamma^1_0=0$ to one
$$
s_{n+1} \colon \frac{k[y]}{\langle y^{n+1} \rangle} \ot_k A \to A \ot_k \frac{k[y]}{\langle y^{n+1} \rangle},
$$
lies in the Hochschild cohomology of $A$ with coefficients in a suitable bimodule. We will call these non commutative polynomial extensions {\em upper triangular}. An intersting fact of these extensions is that the evaluation in $y=0$ is an algebra homomorphism from $A\ot_s k[y]/\langle y^n \rangle$ to $A$. As we point out in Remark~\ref{twisted extensions by power series}, Theorem~\ref{obstruccion en la cohomologia} can also be used to construct a type of non commutative extensions of an algebra $A$ by power series, that we name upper triangular formal extensions of $A$. In order to compare this construction with the formal deformations of $A$ we first note that the power series $k$-algebra $A[[y]]$ has the following properties:

\begin{enumerate}

\smallskip

\item The canonical inclusion $k[[y]]\hookrightarrow A[[y]]$ is a morphism of unitary $k$-algebras and the right $k[[y]]$-module structure on $A[[y]]$ induced by this map is the usual one.

\smallskip

\item The canonical inclusion $A\hookrightarrow A[[y]]$ is a morphism of unitary $k$-algebras and the left $A$-modulo structure on $A[[y]]$ induced by this map is the usual one.

\smallskip

\item The canonical surjection $A[[y]]\to A$ is a morphism of unitary $k$-algebras.

\smallskip

\item The multiplication map $A[[y]]\times A[[y]]\to A[[y]]$ is $k[[y]]$-bilinear.

\smallskip

\end{enumerate}
Let $A_y$ be the underlying $k$-module of $A[[y]]$. The formal deformations of $A$ with unit $1$ are the associative unitary $k$-algebra structures on $A_y$ that satisfy conditions~(1), (3) and~(4), while the upper triangular formal extensions of $A$ are the associative unitary $k$-algebra structures on $A_y$ that satisfy conditions~(1), (2) and~(3).

\smallskip

From now on we assume implicitly that all the maps are $k$-linear maps, all the algebras are over $k$, and the tensor product over $k$ is denoted by $\ot$, without any subscript.

\smallskip

The paper is organized as follows: in Section~1) we make a quick review of the basic general properties of twisted tensor products and twisting maps, we determine necessary and sufficient  conditions for a family of maps $\gamma^r_j\colon A\to A$ ($0\le j,r < n$), in order that the map
$$
s\colon \frac{k[y]}{\langle y^n \rangle} \ot A \to A \ot \frac{k[y]}{\langle y^n \rangle},
$$
defined by the formula~\eqref{eq1intro}, is a twisting map, and we introduce a canonical representation of an arbitrary non commutative truncated polynomial extension $A\ot_s k[y]/\langle y^n\rangle$, of an algebra $A$, in the matrix algebra $M_n(A)$. In Section~2), we study a broad family of non commutative truncated polynomial extensions, which includes those with $\gamma^1_0=0$. In Section~3) we classify the non commutative truncated polynomial extensions with $\gamma^1_0\ne 0$ that satisfy a few natural conditions. Finally, in Section~4), we consider the non commutative truncated polynomial extensions with $\gamma^1_0 = 0$. These can be constructed inductively. For this, the main tool is Theorem~\ref{obstruccion en la cohomologia}. Using it, we obtain several families of these sort of extensions. In particular, all extensions of a truncated polynomial algebra $k [x]/\langle x^m\rangle$ satisfying $ s(y\ot x)\in xk[x]/\langle x^m\rangle\ot yk[y]/\langle y^n\rangle$.

\section{Some basic facts}\label{Some basic facts}
This section is divided in two parts. In the first one, we review the definitions of twisted tensor products and twisting maps, and we establish some of the basic results about these structures. For the proofs we refer to \cite{C-S-V}, \cite{VD-VK} and \cite{C-I-M-Z}. Recall from the introduction that a non commutative truncated polynomial extension of an algebra $A$ is a twisted tensor product $A\ot_s k[y]/\langle y^n \rangle$. In the second one, we start the study of these extensions, by determining  the conditions that a family of maps $\gamma^r_j\colon A\to A$ ($0\le j,r < n$) must fulfill in order that the map
$$
s\colon \frac{k[y]}{\langle y^n \rangle} \ot A \to A \ot \frac{k[y]}{\langle y^n \rangle},
$$
given by
$$
s(y^r\otimes a) = \sum_{j=0}^{n-1}\gamma_j^r(a)\otimes y^j,
$$
is a twisting map.

\subsection{General remarks}
Let $A$ and $C$ be algebras. Let $\mu_A$, $\eta_A$, $\mu_C$ and $\eta_C$ be the multiplication and unit maps of $A$ and $C$, respectively. A {\em twisted tensor product} of $A$ with $C$ is an algebra structure on the $k$-module $A\ot C$, such that the canonical maps
$$
i_A\colon A\to A\ot C\quad\text{and}\quad i_C\colon C\to A\ot C
$$
are algebra homomorphisms and $\mu\xcirc (i_A\ot i_C) = \ide_{A\ot C}$, where $\mu$ denotes the multiplication map of the twisted tensor product.

\smallskip

Assume we have a tensor product of $A$ with $C$. Then, the map
$$
s\colon C\ot A\to A\ot C,
$$
defined by $s: = \mu \xcirc (i_C\ot i_A)$, satisfies:

\begin{enumerate}

\smallskip

\item $s\xcirc (\eta_C\ot A) = A\ot \eta_C$ and $s\xcirc (C\ot \eta_A) = \eta_A\ot C$,

\smallskip

\item $s\xcirc (\mu_C\ot A) = (A\ot \mu_C)\xcirc (s\ot C)\xcirc (C\ot s)$,

\smallskip

\item $s\xcirc (C\ot \mu_A) = (\mu_A\ot C)\xcirc (A\ot s)\xcirc (s\ot A)$.

\smallskip

\end{enumerate}
A map satisfying these conditions is called a {\em twisting map}. Conversely, if
$$
s\colon C\ot A\to A\ot C
$$
is a twisting map, then $A\ot C$ becomes a twisted tensor product via
$$
\mu_s := (\mu_A\ot\mu_C)\xcirc (A\ot s \ot C).
$$
This algebra will be denoted $A\ot_s C$. Furthermore, these constructions are inverse to each other.

\smallskip

The following result is useful in order to check that a map $s\colon C\ot A\to A\ot C$ is a twisting map, and will be used implicitly in this paper.

\begin{proposition}\label{proposicion 1.1} Let $s\colon C\ot A\to A\ot C$ be a map satisfying conditions~(1) and~(2). If $(c_i)_{i\in I}$ generates $C$ as an algebra and
$$
s(c_i\ot aa') = (\mu_A\ot C)\xcirc (A\ot s)\xcirc (s\ot A)(c_i\ot a\ot a')
$$
for all $a,a'\in A$ and each index $i$, then $s$ is a twisting map.
\end{proposition}

\subsection{Non commutative truncated polynomial extensions}
In the sequel we fix $C:= k[y]/\langle y^n\rangle$. Let $A$ be a $k$-algebra and $s\colon C\otimes A\to A\otimes C$ a $k$-linear map. The equations
$$
s(y^r\otimes a) = \sum_{j=0}^{n-1}\gamma_j^r(a)\otimes y^j
$$
define $k$-linear maps $\gamma_j^r\colon A\to A$ for $0\le j,r< n$. Moreover, we put $\gamma_j^r:=0$ if $r\ge n$ and $0\le j<n$. Note that the $\gamma_j^r $'s are defined for $r\ge 0$ and $0\le j <n$.

\begin{proposition}\label{equivalencias de ser un twisting map} The following  assertions are equivalent:

\begin{enumerate}

\item The map $s$ is a twisting map.

\item

\begin{enumerate}

\item $\gamma_j^0=\delta_{j0}\ide$.

\item $\gamma_j^r(1)=\delta_{jr}$.

\item For $j< n$ and $0< r< n$,
$$
\gamma_j^r(ab)=\sum_{i=0}^{n-1}\gamma_i^r(a)\gamma_j^i(b).\qquad\text{(Product law)}
$$

\item For $j<n$, $r>1$ and $0<i<r$,
$$
\gamma_j^r = \sum_{l=0}^j\gamma_l^i\xcirc\gamma_{j-l}^{r-i}.\qquad\text{(Composition law)}
$$

\end{enumerate}

\item

\begin{enumerate}

\item $\gamma_j^0=\delta_{j0}\ide$.

\item $\gamma_j^1(1)=\delta_{j1}$.

\item For $j<n$,
$$
\gamma_j^1(ab)=\sum_{i=0}^{n-1}\gamma_i^1(a)\gamma_j^i(b).
$$

\item For $j<n$ and $r>1$,
$$
\gamma_j^r =\sum_{l=0}^j\gamma_l^1\xcirc\gamma_{j-l}^{r-1}.
$$

\end{enumerate}

\end{enumerate}

\end{proposition}

\begin{proof} (1)~$\Leftrightarrow$~(2)\enspace We know that $s$ is a twisting map if and only if

\begin{enumerate}

\item[(a')] $s(1\otimes a)=a\otimes 1$,

\item[(b')] $s(y^r\ot 1)=1\ot y^r$,

\item[(c')]$s(y^r\otimes ab)=(\mu_A\otimes C)\xcirc (A\otimes s)\xcirc (s\otimes A) (y^r\otimes a\otimes b)$,

\item[(d')] $s(y^ry^t\otimes a)=(A\otimes\mu_C)\xcirc (s\otimes C)\xcirc (C\otimes s) (y^r\otimes y^t\otimes a)$,

\end{enumerate}
for $0<r,t< n$ and $a,b\in A$. But a direct computation shows that~(a')~$\Leftrightarrow$~(2)(a), (b')~$\Leftrightarrow$~(2)(b), (c')~$\Leftrightarrow$~(2)(c) and (d')~$\Leftrightarrow$~(2)(d).

\smallskip

\noindent (2)~$\Rightarrow$~(3)\enspace This is trivial.

\smallskip

\noindent (3)~$\Rightarrow$~(2)\enspace First note that~(2)(b) follows immediately from~(3)(a), (3)(b) and~(3)(d). We now prove that condition~(2)(d) holds. For $i=1$ and $r>1$ this is the same as~(3)(d). We suppose that~(2)(d) is true for a fixed $i>0$ and all $r>i$, and we prove it for $i+1$ and all $r>i+1$. Fix $r>i+1$. Then
\allowdisplaybreaks
\begin{align*}
\gamma^r_j &=\sum_{h=0}^j\gamma^1_h\xcirc\gamma^{r-1}_{j-h}&&\text{by~(3)(d)}\\
&= \sum_{h=0}^j\sum_{u=0}^{j-h}\gamma^1_h\xcirc \gamma^i_u\xcirc \gamma^{r-i-1}_{j-h-u} && \text{by inductive hypothesis}\\
&=\sum_{h=0}^j\sum_{l=h}^j\gamma^1_h\xcirc \gamma^i_{l-h}\xcirc \gamma^{r-i-1}_{j-l}&& \text{setting $l:=u+h$}\\
&= \sum_{l=0}^j\sum_{h=0}^l\gamma^1_h\xcirc \gamma^i_{l-h}\xcirc \gamma^{r-i-1}_{j-l}\\
&= \sum_{l=0}^j \gamma^{i+1}_l\xcirc  \gamma^{r-i-1}_{j-l}.&&\text{by~(3)(d)}
\end{align*}
So~(2)(d) is true. It remains to check that~(2)(c) is also true. For $r=1$ it is the same as~(3)(c). Suppose~(2)(c) holds for a fixed $r$ with $1\le r<n-1$. Then
\begin{align*}
\gamma_j^{r+1}(ab)&=\sum_{l=0}^j\gamma_l^1\bigl(\gamma_{j-l}^r(ab)\bigr) &&\text{by~(3)(d)}\\
&=\sum_{l=0}^j\gamma_l^1\left(\sum_{i=0}^{n-1}\gamma_i^r(a)\gamma_{j-l}^i(b)\right)&&\text{by inductive hypothesis}\\
&=\sum_{l=0}^j\sum_{i=0}^{n-1}\sum_{m=0}^{n-1}\gamma_m^1\left(\gamma_i^r(a)\right)\gamma_l^m \left(\gamma_{j-l}^i(b) \right)&&\text{by~(3)(c)}\\
&=\sum_{i=0}^{n-1}\sum_{m=0}^{n-1}\gamma_m^1\left(\gamma_i^r(a)\right)\sum_{l=0}^j\gamma_l^m \left(\gamma_{j-l}^i(b) \right)\\
&=\sum_{m=0}^{n-1}\sum_{i=0}^{n-1}\gamma_m^1\left(\gamma_i^r(a)\right)\gamma_j^{m+i}(b) && \text{by~(2)(a) and (2)(d)}\\
&=\sum_{m=0}^{n-1}\sum_{u=m}^{n-1}\gamma_m^1\left(\gamma_{u-m}^r(a)\right)\gamma_j^u(b) && \parbox{1.5in}{setting $u:=m+i$, since $\gamma_j^u=0$ for $u\ge n$}\\
&=\sum_{u=0}^{n-1}\sum_{m=0}^u\gamma_m^1\left(\gamma_{u-m}^r(a)\right)\gamma_j^u(b)\\
&=\sum_{u=0}^{n-1}\gamma_u^{r+1}(a)\gamma_j^u(b).&&\text{by~(3)(d)}
\end{align*}
This finishes the proof.
\end{proof}

In the following three remarks we assume that $s\colon C\ot A\to A\ot C$ is a twisting map.

\begin{remark}\label{ker gamma^1_0} Let $B:=\ker \gamma^1_0$. By items~(3)(a), (3)(b), the Product law and the Composition law, $B$ is a subalgebra of $A$ and $\gamma^1_0$ is a right $B$-linear map. Consequently, if $b'b=bb'=1$ and $b\in B$, then $b'\in B$.
\end{remark}

\begin{remark} The Composition Law is valid for $r\ge 0$ and $0\le i\le r$. This follows immediately from items~(2)(a) and~(2)(d), and will be used freely throughout the paper.
\end{remark}

\begin{remark}\label{formula para gamma_k^j} From item~(3)(d) of the above proposition it follows easily by induction on $r$ that
\begin{equation}
\gamma^r_j=\sum_{u_1,\dots, u_r\ge 0 \atop u_1+\cdots+u_r=j}\gamma^1_{u_1}\xcirc\cdots\xcirc \gamma^1_{u_r}\label{eq formula para gamma_k^j}
\end{equation}
for all $r\ge 1$. In particular $\gamma^r_0=\gamma^1_0\xcirc \cdots \xcirc \gamma^1_0$ ($r$ times).
\end{remark}

\begin{corollary}\label{gamma como suma} For each $0\le j < n$, let $\gamma_j^1\colon A\to A$ be a $k$-linear map satisfying \mbox{$\gamma^1_j(1)\!=\delta_{1j}$.} Set $\gamma^0_j:=\delta_{0j} \ide$ and
$$
\gamma^r_j:=\sum_{u_1,\dots, u_r\ge 0\atop u_1+\cdots+u_r=j}\gamma^1_{u_1}\xcirc\cdots\xcirc \gamma^1_{u_r}\quad \text{for $r>1$ and $j<n$.}
$$
If $\gamma_j^n=0$ for all $j< n$ and
$$
\gamma_j^1(ab)=\sum_{i=0}^{n-1}\gamma_i^1(a)\gamma_j^i(b)\qquad\text{for $a,b\in A$ and $j<n$,}
$$
then the maps $\gamma_j^r$ satisfy the equivalent conditions of Proposition~\ref{equivalencias de ser un twisting map}.
\end{corollary}

\begin{proof} By hypothesis we know that~(3)(b) and~(3)(c) of Proposition~\ref{equivalencias de ser un twisting map} hold. Moreover, by the definition of the $\gamma^r_j$'s, it is clear that the maps $\gamma^r_j$ satisfy items~(3)(a) and~(3)(d) of the same proposition, and that $\gamma^r_j=0$ for $r\ge n$.
\end{proof}

\begin{remark} Notice that when $\gamma^1_0=0$, then the condition $\gamma^n_j=0$ for $j<n$ (in the above corollary) is automatically satisfied.
\end{remark}

\begin{example}\label{Ore} Assume that $\gamma^1_j = 0$ for all $j>1$. Then, from formula~\ref{eq formula para gamma_k^j} it follows immediately that $\gamma^r_j = 0$ for all $j>r$. In this case Conditions~(3)(a) and~(3)(c) becomes
\begin{align*}
& \gamma^1_0(1) = 0,\\
& \gamma^1_1(1) = 1,\\
& \gamma^1_0(ab) = \gamma^1_0(a)b+\gamma^1_1(a)\gamma^1_0(b)\\
\intertext{and}
& \gamma^1_1(ab) = \gamma^1_1(a)\gamma^1_1(b).
\end{align*}
In other words $\gamma^1_1$ is an algebra endomorphism and $\gamma^1_0$ is a $\gamma^1_1$-derivation. So, by Corollary~\ref{gamma como suma}, in order to have a twisting map, we must require that
\begin{equation}
\gamma^n_j = \sum_{u_1,\dots, u_r\in\{0,1\}\atop u_1+\cdots+u_n=j}\gamma^1_{u_1}\xcirc\cdots\xcirc \gamma^1_{u_r} = 0 \quad \text{for all $j<n$.}\label{extensionesdeOre}
\end{equation}
For example, for $n = 2$ this becomes
$$
\gamma^1_0\xcirc \gamma^1_0 = 0\quad\text{and}\quad \gamma^1_0\xcirc \gamma^1_1 + \gamma^1_1\xcirc \gamma^1_0 = 0.
$$
We will call these twisting maps and their corresponding twisted products, {\em lower triangular}. Note that there is a close analogy of these twisted products with the classical Ore extensions.
\end{example}

Associated with a twisting map
$$
s\colon C\otimes A\to A\otimes C
$$
we have the matrix $M\in\M_n(\End_k(A))$ given by
\vspace{2\jot}
$$
\vspace{2\jot}
M:=\begin{pmatrix}
\ide & 0 &\dots & 0\\
\gamma_0^1&\gamma_1^1&\dots&\gamma_{n-1}^1\\
\vdots &&\ddots &\vdots\\
\gamma_0^{n-1}&\gamma_1^{n-1}&\dots&\gamma_{n-1}^{n-1}
\end{pmatrix}.
$$
Moreover, for $a\in A$ we define the matrix $M(a)\in\M_n(A)$ as the evaluation of $M$ in $a$. That is
$$
M(a)_{ij}:=\gamma^i_j(a)\quad\text{($0\le i,j< n$)}
$$

\begin{corollary}\label{M es multiplicativa} The matrices $M(a)$ fulfill:

\begin{enumerate}

\item $M(1)=Id$.

\item $M(ab)=M(a)M(b)$.

\end{enumerate}

\end{corollary}

\begin{proof} This follows from the Product law and the fact that $\gamma^r_j(1)=\delta_{j1}$.
\end{proof}

\begin{theorem} The formulas $\varphi(a):=M(a)$ for $a\in A$, and
$$
\vspace{2\jot}
\varphi(y):=\begin{pmatrix}
0&1&\cdots &0\\
\vdots&\vdots&\ddots&\vdots\\
0&0&\cdots &1\\
0&0&\cdots &0
\end{pmatrix}\qquad\text{(the nilpotent Jordan matrix $J_0$),}
$$
define a faithful representation $\varphi\colon A\ot_s C\to M_n(A)$.
\end{theorem}

\begin{proof} Since $\varphi(y)^n=0$, in order to check that $\varphi$ defines an algebra map, we only need to verify that
$$
\varphi(y)\varphi(a)=\varphi(\gamma^1_0(a))+\varphi(\gamma^1_1(a))\varphi(y)+\cdots+ \varphi(\gamma^1_{n-1}(a)) \varphi(y)^{n-1}.
$$
But note that
$$
\bigl(J_0 M(b)\bigr)_{ij} = \begin{cases} M(b)_{i+1,j}&\text{for $i<n-1$,}\\ 0 & \text{otherwise,}\end{cases}
$$
and
$$
\bigl( M(b) J_0^u\bigr)_{ij} = \begin{cases} M(b)_{i,j-u}&\text{for $j\ge u$,}\\ 0& \text{otherwise,}
\end{cases}
$$
and so
\allowdisplaybreaks
\begin{align*}
\bigl(\varphi(y)\varphi(a)\bigr)_{ij}
&=\bigl(J_0 M(a)\bigr)_{ij}\\
&=\gamma^{i+1}_j(a)\\
&=\sum_{u=0}^j\gamma^i_{j-u}\bigl(\gamma_u^1(a)\bigr)\\
&=\sum_{u=0}^j M\bigl(\gamma_u^1(a)\bigr)_{i,j-u}\\
&=\sum_{u=0}^{n-1}\bigl(M\bigl(\gamma_u^1(a)\bigr)J_0^u\bigr)_{ij}\\
&=\sum_{u=0}^{n-1}\bigl(\varphi(\gamma^1_u(a))\varphi(y)^u\bigr)_{ij},
\end{align*}
where the second equality is valid also in the case $i=n-1$, since
$$
\bigl(J_0 M(a)\bigr)_{n-1,j}=0 =\gamma^n_j(a).
$$
The injectivity follows from the fact that the composition of $\varphi$ with the surjection onto the first row gives the canonical linear isomorphism $A\ot_s C\to A^n$.
\end{proof}

\subsection{Simplicity of the noncommutative truncated polynomial extensions} Next we characterize the simple twisted tensor products $A\ot_s C$.

\begin{proposition}\label{simplicidad} A twisted tensor products $A\ot_s C$ is simple if and only if $A a \gamma^{n-1}_0(A) = A$ for all $a\in \ker \gamma^1_0\setminus \{0\}$.
\end{proposition}

\begin{proof} Let $D:=A\ot_s C$ and $B:=\ker \gamma^1_0$. By definition $D$ is simple if and only if $DPD = D$ for all $P\in D\setminus \{0\}$. Write
$$
P:= a_iy^i+a_{i+1}y^{i+1}+\dots+a_{n-1}y^{n-1}
$$
with $a_i\ne 0$. Since
$$
a_iy^{n-1} = P y^{n-i-1},\quad ya_iy^{n-1} = \gamma^1_0(a_i)y^{n-1},
$$
and, by Remark~\ref{formula para gamma_k^j}, the map $\gamma^1_0$ is nilpotent, in order to check that $D$ is simple, it is necessary and sufficient to verify that
$$
Day^{n-1}D=D\quad\text{for all $a\in B\setminus\{0\}$.}
$$
Let $Q:= \sum b_iy^i$ and $R:= \sum c_iy^i$ in $D$. Using that $a\in B$, it is easy to see that
$$
Q ay^{n-1} R = b_0 a \gamma^{n-1}_0(c_0) + Sy\quad\text{where $S\in D$.}
$$
From this it follows immediately that if $D$ is simple, then
$$
A a \gamma^{n-1}_0(A) = A\quad\text{for all $a\in B\setminus \{0\}$.}
$$
Conversely, if this is true, then there exist $Q_1,R_1,\dots,Q_t,R_t$ such that
$$
\sum_{i=1}^t Q_i ay^{n-1} R_i = 1 + Sy\quad\text{where $S\in D$.}
$$
Hence, in order to finish the proof it suffices to note that if
$$
1+Sy^i\in Day^{n-1}D,
$$
then
$$
1-S^2y^{2i} = 1+Sy^i - S(1+Sy^i)y^i
$$
also belongs to $Day^{n-1}D$.
\end{proof}

\section{A family of twisting maps}
Recall that $C:=k[y]/\langle y^n\rangle$. Let $A$ be a $k$-algebra. The aim of this section is to study the broad family of twisting maps
$$
s\colon C\ot A\to A\ot C
$$
satisfying the following conditions:

\begin{itemize}

\smallskip

\item[A1)] {\em There exist $1\le h\le n$ and  $x\in A$ such that $\gamma^h_0=0$ and $q:=\gamma^{h-1}_0(x)$ is right cancelable},

\smallskip

\item[A2)] {\em $\gamma^1_0$ is an endomorphism of $B$-bimodules, where $B:=\ker \gamma^1_0$}.

\medskip

\end{itemize}
Actually Condition~A1) is used throughout all the section, but Condition~A2) is not used until Lemma~\ref{propiedadpolinomios}.

\begin{remark}\label{A2 es automatica si A es conmutativo} Condition~A1) is always fulfilled if $A$ is a cancelative ring. On the other hand, by Remark~\ref{ker gamma^1_0}, we know that $\gamma^1_0$ is a right $B$-linear map. So, Condition~A2) is automatically fulfilled if $B$ is included in the center of $A$. Hence conditions~A1) and~A2) are both satisfied if $A$ is a commutative domain.
\end{remark}

\begin{remark}\label{si q es inversible puede tomars=1} For some results we will need to ask that $q$ is invertible. Since $\gamma^{h-1}_0(A)$ is a right ideal of $B$, in this case we can assume that $q=1$.
\end{remark}

\begin{remark} The family that we are going to consider includes all the twisting maps with $\gamma^1_0=0$. However the results we establish in this section only are relevant when $\gamma^1_0\ne 0$.
\end{remark}

In the sequel, for every $a\in $A we let $M_{(h)}(a)$ denote the $h\times h$-submatrix of $M(a)$ formed by the first $h$ rows and columns of $M(a)$, and we fix both,  $x$ and $q$. Note that, by the Composition Law, $q\in B$.

\begin{lemma}\label{bloque} Let $\gamma_j^i\colon A\to A$ ($0\le j< n$ and $r\ge 0$) be a family of maps satisfying the Composition law and~A1). Assume that
$$
\gamma^h_0(aa') = \sum_{i=0}^{n-1}\gamma_i^h(a)\gamma_0^i(a')\quad\text{for all $a,a'\in A$.}
$$
Then, for each $j=0,\dots,h-1$ and $i\ge h$, the map $\gamma_j^i=0$. Consequently,
$$
M_{(h)}(ab)=M_{(h)}(a)M_{(h)}(b)\qquad\text{for all $a,b\in A$.}
$$
\end{lemma}

\begin{proof} By the Composition law it suffices to check this for $i=h$. For $0\le j <h$, let $b_j:= \gamma_0^{h-1-j}(x)$. Again by the Composition law,
\begin{equation}
\gamma_0^r(b_j) =\begin{cases}q &\text{if $r=j$},\\0&\text{if $r>j$.}\end{cases}\label{eqq6}
\end{equation}
Let $a$ be an arbitrary element of $A$. Since,
$$
0=\gamma^h_0(ab_1)=\sum_{i=0}^{n-1}\gamma_i^h(a)\gamma_0^i(b_1)=\gamma^h_1(a)q,
$$
we have $\gamma^h_1(a)=0$. Then,
$$
0=\gamma^h_0(ab_2)=\sum_{i=0}^{n-1}\gamma_i^h(a)\gamma_0^i(b_2)=\gamma^h_2(a)q,
$$
and so $\gamma^h_2(a)=0$, etcetera.
\end{proof}

\begin{proposition}\label{escalera} Assume that the hypothesis of the previous lemma are fulfilled and let $0< l\le\lfloor n/h\rfloor$. Then, for each $i\ge lh$ and $j< lh$, the map $\gamma^i_j$ vanishes.
\end{proposition}

\begin{proof} We proceed by induction on $l$. For $l=1$ the result is the previous lemma. Assuming it is true for $l\ge 1$,
$$
\gamma^i_j=\sum_{u=0}^j\gamma^h_u\xcirc\gamma^{i-h}_{j-u}=0,
$$
for each $i\ge (l+1)h$ and $j< (l+1)h$, as we want, since in each summand $\gamma^h_u\xcirc\gamma^{i-h}_{j-u}$ one of the factors vanishes.
\end{proof}

The previous result can be rephrased by saying that the matrix $M$ has the following shape:
\[
 M=\left(
\begin{array}{l@{} l@{} c}
 \begin{array}{|cccc}
     \ide & 0  & \dots & 0\\
     \gamma^1_0 & \gamma^1_1 & \dots & \gamma^1_{h-1}\\
     \vdots & \vdots & \ddots & \vdots\\
     \gamma^{h-1}_0 & \gamma^{h-1}_1 & \dots & \gamma^{h-1}_{h-1}\\
     \hline
 \end{array}
 &
 \begin{array}{cccccccccccc}
    \hdotsfor{12}\\
    \hdotsfor{12}\\
    & & & &  & \phantom{\vdots}&&&&& & \\
    \hdotsfor{12}\\
 \end{array}
 &
 \begin{array}{c}
    0\\
    \gamma^1_{n-1}\\
    \vdots\\
    \gamma^{h-1}_{n-1}\\
 \end{array}\\
 \begin{array}{cccc}
    0 & 0 & \cdots & 0\\
    \phantom{\gamma^{h-1}_0} & \phantom{\gamma^{h-1}_1} & \phantom{\cdots} & \phantom{\gamma^{h-1}_{h-1}}\\
    & & &
 \end{array} &
 \begin{array}{|cccc}
    \gamma^h_h & \cdots & \gamma^h_{2h-1}\\
    \vdots & \ddots & \vdots\\
    \gamma^{2h-1}_h & \cdots & \gamma^{2h-1}_{2h-1}\\
    \hline
 \end{array}
 \begin{array}{c}
    \, \cdots \\
    \, \cdots \\
 \end{array}
 &
 \begin{array}{c}
    \vdots \\
    \vdots \\
 \end{array} \\
 &
 \begin{array}{ccc}
    0 &  \cdots & 0\\
    \phantom{\gamma^{h-1}_1} & \phantom{\cdots} & \phantom{\gamma^{h-1}_{h-1}}
 \end{array}
 \phantom{n,}
 \begin{array}{|c}
      \\
      \\
 \end{array}
 &
 \\

 &&\\
\begin{array}{cc}
    \vdots & \\
    0 & \cdots \\
 \end{array}
& \begin{array}{cc}
    & \\
    & \\
 \end{array}
 &
 \begin{array}{c}
     \\
    \gamma_{n-1}^{n-1} \\
 \end{array}

\end{array}
 \right)
 \]

\medskip

\begin{proposition}\label{variante de gamma como suma} For each $0\le j < n$, let $\gamma_j^1\colon A\to A$ be a $k$-linear map satisfying \mbox{$\gamma^1_j(1)\!=\delta_{1j}$.} Set $\gamma^0_j:=\delta_{0j} \ide$ and
\begin{equation}\label{aa}
\gamma^r_j:=\sum_{u_1,\dots, u_r\ge 0\atop u_1+\cdots+u_r=j}\gamma^1_{u_1}\xcirc\cdots\xcirc \gamma^1_{u_r}\quad \text{for $r>1$ and $j<n$.}
\end{equation}
If:

\begin{enumerate}

\smallskip

\item There exists $1\le h \le n$ and $x\in A$ such that $\gamma^h_0=0$ and $q:=\gamma^{h-1}_0(x)$ is right cancelable,

\smallskip

\item $\gamma_j^n=0$ for all $j \ge \lfloor n/h\rfloor h$,

\smallskip

\item $\gamma_j^1(ab)=\sum_{i=0}^{n-1}\gamma_i^1(a)\gamma_j^i(b)$ for $a,b\in A$ and $j<n$,

\smallskip

\end{enumerate}
then the family of maps $(\gamma_j^r)_{0\le j,r<n}$ defines a twisting map $s\colon C\ot A\to A\ot C$.
\end{proposition}

\begin{proof} By the hypothesis, and the definition of the maps $\gamma^0_j$, it is evident that items~(3)(b) and~(3)(c) of Proposition~\ref{equivalencias de ser un twisting map} hold. Moreover, using~\eqref{aa} it is easy to check that the maps $\gamma^r_j$ satisfy the Composition law (and, in particular, item~(3)(d) of the same proposition). Hence, in order to finish the proof we only need to check that $\gamma^n_j=0$ for $0\le j< n$. But

\begin{itemize}

\smallskip

\item[(i)] From item~(1) and the Composition law it follows that $\gamma^r_0=0$ for $r\ge h$,

\smallskip

\item[(ii)] Using (i) and arguing as in the last part of the proof of Proposition~\ref{equivalencias de ser un twisting map}, we check that
$$
\gamma^r_0(aa')=\sum_{i=0}^{n-1} \gamma^r_i(a)\gamma^i_0(a')\quad\text{for all $r> 0$ and $a,a'\in A$.}
$$
\smallskip

\end{itemize}
Hence $\gamma^n_j=0$ for $j < \lfloor n/h\rfloor h$, by Proposition~\ref{escalera}. By item~(2) this finish the proof.
\end{proof}

\begin{proposition}\label{gamma^{lh+r}_lh(b_j)} Under the hypothesis of Lemma~\ref{bloque}, the elements $b_j\!:=\!\gamma_0^{h-1-j}(x)$ ($j=0,\dots,h-1$), introduced in the proof of that result, satisfy
$$
\gamma^{lh+r}_{lh}(b_j)=\begin{cases} \gamma^{lh}_{lh} (b_{j-r}) &\text{if $r\le j$,}\\ 0 & \text{otherwise,}
\end{cases}
$$
for $0\le l \le \lfloor \frac{n-1}{h} \rfloor$.
\end{proposition}

\begin{proof} By the Composition law and Proposition~\ref{escalera},
$$
\gamma^{lh+r}_{lh}(b_j)= \sum_{u=0}^{lh}\gamma_{u}^{lh}\bigl(\gamma^r_{lh-u}(b_j)\bigr) = \gamma^{lh}_{lh} \bigl(\gamma^r_0(b_j)\bigr).
$$
The assertion follows now from the definition of the $b_i$'s and equality~\eqref{eqq6}.
\end{proof}

\begin{proposition}\label{h divide a n} Let $s\colon C\ot A \to A \ot C$ be a twisting map that satisfies~A1). If $q$ is invertible, then $h$ divides $n$.
\end{proposition}

\begin{proof} By Remark 2.1 we can assume $q=1$. Let $j\in \{0,\dots, h-1\}$. By Proposition~\ref{gamma^{lh+r}_lh(b_j)} and Proposition~\ref{equivalencias de ser un twisting map}, we know that $\gamma^{lh+j}_{lh}(b_j)=\gamma^{lh}_{lh}(1)=1$. But, if $h\nmid n$, then the case $j=n-lh$, with $l=\lfloor\frac{n}{h}\rfloor$, leads to $\gamma^n_{lh}(b_j)=1$, which is impossible, since $\gamma^n_{lh}=0$.
\end{proof}

\begin{lemma}\label{basis} Let $D$ be a $k$-algebra, and let $g\colon D\to D$ be a $k$-linear map. Assume that $g^h=0$ and that there exists $x\in D$ such that $q:=g^{h-1}(x)$ is invertible. Suppose also that $E:=\ker g$ is a $k$-subalgebra of $D$ and $g$ is a right $E$-linear map. Then $D$ is a free right $E$-module of rank $h$. Moreover  $\mathfrak{B}:= \{x,g(x),\dots,g^{h-1}(x)\}$ is a basis.
\end{lemma}

\begin{proof} Consider a null combination
$$
\sum_{i=0}^{h-1}g^i(x)\lambda_i=0,
$$
with coefficients in $E$. Applying $g^{h-1}$ to both sizes of this equality, we see that $g^{h-1}(x)\lambda_0 =q\lambda_0=0$. Hence, $\lambda_0=0$. Now, applying successively $g^{h-2},\dots, g^1$, we get $\lambda_1=0, \dots, \lambda_{h-1}=0$. So $\mathfrak{B}$ is linearly independent. It remains to check that $\mathfrak{B}$ generates $D$ as a right $E$-module. Note that $q^{-1}\in E$, because $g(q^{-1})q=g(1)=0$. We will prove by induction on $i$ that there exist $\lambda_0,\dots, \lambda_{i-1}\in E$, such that
\begin{equation}
g^{h-i}(a) = \sum_{j=0}^{i-1} g^{h-i+j}(x)\lambda_j \quad \text{for all $a\in D$ and $i=0,\dots,h$.}\label{eqq7}
\end{equation}
The case $i=0$ is trivial, since $g^h(a) = 0$ and on the right side of~\eqref{eqq7} we have the empty sum (which gives $0$). Assume that~\eqref{eqq7} holds for a fixed $i<h$ and set
\begin{equation}
a_i:= a - x\lambda_0 - g(x)\lambda_1-\cdots-g^{i-1}(x)\lambda_{i-1}.\label{eqq8}
\end{equation}
From~\eqref{eqq7} it follows immediately that $g^{h-i}(a_i) = 0$. Hence $g^{h-i-1}(a_i)\in E$, which implies that $\lambda_i = q^{-1}g^{h-i-1}(a_i)\in E$. Consequently, by~\eqref{eqq8},
$$
g^{h-i-1}(a) = g^{h-i-1}(x)\lambda_0 + g^{h-i}(x)\lambda_1+\cdots+g^{h-2}(x)\lambda_{i-1}+ g^{h-1}(x)\lambda_i,
$$
since $g^{h-i-1}(a_i) = q\lambda_i = g^{h-1}(x)\lambda_i$.
\end{proof}

\begin{theorem}\label{A es libre sobre B} Let $s\colon C\ot A \to A \ot C$ be a twisting map that satisfies~A1). If $q$ is invertible, then $A$ is a right free $B$-module. Furthermore
$$
\mathfrak{B}:=\{x,\gamma^1_0(x),\dots,\gamma^{h-1}_0(x)\}
$$
is a basis.
\end{theorem}

\begin{proof} Apply the previous lemma with $D=A$, $E=B$ and $g=\gamma^1_0$.
\end{proof}

\begin{corollary}\label{para dimension finita} Let $k$ be a field and let $A$ be a finite dimensional $k$-algebra. If there exists a twisting map $s\colon C\ot A\to A\ot C$ that satisfies Condition~A1), then $dim_k(A)=h\cdot dim_k(B)$.
\end{corollary}

\begin{proof} This follows from Theorem~\ref{A es libre sobre B}, since in a finite dimensional $k$-algebra each right cancelable element $q$ is invertible.
\end{proof}

With the only exception of Lemma~\ref{diagonales}, in the rest of the results of this section $s$ is a twisting map that satisfies conditions~A1) and~A2).

\begin{lemma}\label{propiedadpolinomios} If $b\in B$, then $M_{(h)}(b)=bI_h$.
\end{lemma}

\begin{proof} When $h=1$, then $M_{(h)}=\gamma^0_0=\ide_A$, and the result is trivial. Assume that $h>1$. Note that $\gamma^1_0(b)=0$ implies $\gamma_0^i(b)=0$ for all $i>0$. Let $b_j$ ($j=1,\dots,h-1$) be as in Lemma~\ref{bloque}. Consider the matrix
$$
M_{(h)}(b_1):=
\begin{pmatrix}
b_1 & 0 &\dots & 0\\
q &\gamma^1_1(b_1) &\dots &\gamma^1_{h-1}(b_1)\\
0 &\vdots &\ddots &\vdots\\
0 &\gamma^{h-1}_1(b_1) &\dots &\gamma^{h-1}_{h-1}(b_1)
\end{pmatrix}.
$$
By Condition~A2) and Lemma~\ref{bloque},
$$
bq= b \gamma^1_0(b_1) = \gamma^1_0(bb_1) = \bigl(M_{(h)}(b)M_{(h)}(b_1)\bigr)_{10} = \gamma^1_1(b)q,
$$
and so $\gamma^1_1(b)=b$, by Condition~A1). The same matrix product at the entries $(j,0)$ for $j=2,\dots,h-1$, combined with the facts that $\gamma^j_0(b_1)=0$ and $\gamma^j_0$ is left $B$-linear, yields $\gamma^j_1(b)q=0$, and so $\gamma^j_1(b)=0$. Now, since $\gamma^0_0(b)=\gamma^1_1(b)=b$ and $\gamma^j_i(b)=0$ for $i=0,1$ and $j\ne i$, the equalities
$$
b\gamma^j_0(b_2)=M_{(h)}(bb_2)_{j0}=\bigl(M_{(h)}(b)M_{(h)}(b_2)\bigr)_{j0} \qquad{j=1,\dots,h-1},
$$
give $\gamma^2_2(b)=b$ and $\gamma^j_2(b)= 0$ for $j\ne 2$. Proceeding in the same way successively with $M_{(h)}(b_3),\dots, M_{(h)}(b_{n-1})$, we obtain the desired result.
\end{proof}

\begin{proposition}\label{linealidad de los gamma_ij sobre el nucleo} For $i,j=0,\dots,h-1$, the map $\gamma^i_j\colon A\to A$ is left and right $B$-linear.
\end{proposition}

\begin{proof} We only check the left linearity, since the right one is similar. Let $b\in B$ and $a\in A$. By Lemmas~\ref{bloque} and~\ref{propiedadpolinomios},
$$
\gamma^i_j(ba) =\bigl(M_{(h)}(b)M_{(h)}(a)\bigr)_{ij} =\bigl(bI_hM_{(h)}(a)\bigr)_{ij} = bM_{(h)}(a)_{ij} = b\gamma^i_j(a),
$$
as we want.
\end{proof}

\begin{proposition}\label{triangular} For each $b\in B$ the matrix $M(b)$ is upper triangular. Moreover,
$$
\gamma^{lh}_{lh+u}(b)=\gamma^{lh+1}_{lh+u+1}(b)=\cdots =\gamma^{lh+h-u-1}_{lh+h-1}(b)\quad\text{for $l<\left\lfloor \frac{n}{h}\right\rfloor$ and $u<h$,}
$$
and
$$
\gamma^{\lfloor n/h\rfloor h}_{\lfloor n/h\rfloor h+u}(b)=\cdots =\gamma^{n-u-1}_{n-1}(b)\quad \text{for $u<n-\left\lfloor \frac{n}{h}\right\rfloor h$ if $h$ does not divide $n$.}
$$
\end{proposition}

\begin{proof} In order to check that $M(b)$ is upper triangular it suffices to verify that $\gamma^{lh+i}_{lh+j}(b)=0$ for $l\ge 0$, $0\le j< h$ and $i>j$. If $i\ge h$, this follows from Pro\-position~\ref{escalera}. So, we can assume that $j<i<h$. Let $i':=lh+i$ and $j':=lh+j$. We have
$$
\gamma^{i'}_{j'}(b) = \sum_{v=0}^{i-1}\gamma^{lh}_{j'-v}\bigl(\gamma^i_v(b)\bigr)+\sum_{v=i}^{u'} \gamma^{lh}_{j'-v} \bigl(\gamma^i_v(b)\bigr) =0,
$$
because $\gamma^i_v(b)=0$ for $b\in B$ and $v<i$, by Lemma~\ref{propiedadpolinomios}, and $\gamma^{lh}_{j'-v}=0$ for $v>j$, by Proposition~\ref{escalera}. Now we are going to prove the equalities. Assume first that $l<\lfloor n/h\rfloor$. Let
$$
0\le v < h-u,\quad v':=lh+v\quad\text{and}\quad u':=lh+u+v.
$$
Then
$$
\gamma^{v'}_{u'}(b) =\sum_{\substack{j=0\\ j\ne v }}^{h-1} \gamma_{u'-j}^{lh}\bigl(\gamma_j^v(b)\bigr) +\gamma^{lh}_{lh+u}\bigl(\gamma^v_v(b)\bigr) +\sum_{j=h}^{u'}\gamma_{u'-j}^{lh}\bigl(\gamma_j^v(b)\bigr) =\gamma^{lh}_{lh+u}(b),
$$
since $\gamma^v_j(b)=0$ for $j<h$, $j\ne v$, $\gamma^{lh}_{u'-j} = 0$ for $j\ge h$, and $\gamma^v_v(b)=b$. The case $l=\lfloor n/h\rfloor$ is similar, but we must take $0\le v<n-\lfloor n/h\rfloor h-u$.
\end{proof}

\begin{lemma}\label{diagonales} Let $s\colon C\ot E\to E\ot C$ be a twisting map and let $2\le j_0 <n$. If $\gamma^1_j = \delta_{1j}\ide$ for $j<j_0$, then for all $i<n$ the following facts hold:

\begin{enumerate}

\smallskip

\item $\gamma^i_l = 0$ for $l<i$ and $\gamma^i_i = \ide$.

\smallskip

\item $\gamma^i_{i+l} = 0$ for $0<l<\min(j_0-1, n-i)$.

\smallskip

\item $\gamma^i_{i+j_0-1} = i\gamma^1_{j_0}$ for $i\le n-j_0$.

\end{enumerate}

\end{lemma}

\begin{proof} Note first that item~(1) follows from formula~\eqref{eq formula para gamma_k^j} and the fact that $\gamma^1_0 = 0$ and $\gamma^1_1 = \ide$. We now prove item~(2). Again by formula~\eqref{eq formula para gamma_k^j}
$$
\gamma^i_{i+l} = \sum_{u_1,\dots, u_i\ge 0\atop u_1+\cdots+u_i=i+l} \gamma^1_{u_1}\xcirc \cdots\xcirc \gamma^1_{u_i}
$$
Suppose $\gamma^i_{i+l}\ne 0$. Then some $\gamma^1_{u_1}\xcirc \cdots\xcirc \gamma^1_{u_i}\ne 0$. Since $\gamma^1_j = \delta_{1j}\ide$ for $j<j_0$, each $u_v$ is $1$ or greater or equal than $j_0$. But since $l\ge 1$ and $u_1+\cdots+u_i = i+l$, there is at least one $u_v$ greater or equal than $j_0$. But then $u_1+\cdots+u_i\ge i + j_0-1>i+l$, which is a contradiction. We finally prove item~(3). We proceed by induction on $i$. The case $i=1$ is trivial. Assume that $i>1$ and that the result is valid for $i-1$. By item~(1), we know that $\gamma^{i-1}_{i+j_0-1-v} = 0$ for $v>j_0$. Moreover, $\gamma^1_v = 0$ for $1<v<j_0$. Hence,
$$
\gamma^i_{i+j_0-1} = \sum_{v=0}^{i+j_0-1}\gamma^{i-1}_{i+j_0-1-v}\xcirc \gamma^1_v = \gamma^{i-1}_{i+j_0-2}\xcirc \gamma^1_1 + \gamma^{i-1}_{i-1}\xcirc \gamma^1_{j_0} = (i-1)\gamma^1_{j_0} + \gamma^1_{j_0} = i\gamma^1_{j_0},
$$
where the third equality follows from item~(1) and the inductive hypothesis.
\end{proof}

\begin{proposition}\label{s induces the flip}  Assume that $q$ is invertible, $B$ is included in the center of $A$, and $h$ is greater than $1$ and cancelable in $B$. Then,
$$
s(C\ot B)\subseteq B\ot C\text{ if and only if } s(c\ot b) = b\ot c\text{ for all }c\in C\text{ and }b\in B.
$$
\end{proposition}

\begin{proof} By Remark~\ref{si q es inversible puede tomars=1} we can assume that $q=1$. Suppose that $s(C\ot B)\subseteq B\ot C$, which implies that $\gamma^i_j(B)\subseteq B$ for all $i,j$. By Lemma~\ref{propiedadpolinomios} we  know that $\gamma^1_1(b) = b$ for all $b\in B$. Hence, by items~(1) and (2) of Lemma~\ref{diagonales}, in order to finish the proof its suffices to check that $\gamma^1_j=0$ on B, for all $j\ge 2$. Again by Lemma~\ref{propiedadpolinomios} this is true for $1 < j < h$, and we are going to prove it for $j\ge h$ by induction on $j$. So, we assume that $\gamma^1_j(b)=0$, for all $b\in B$ and $h-1\le j < j_0$, and we consider two cases:

\smallskip

\noindent a) If $j_0=lh$ for some $l\ge 1$, then
\begin{align*}
\quad \gamma^h_{lh}(bx) & =\sum_{j=0}^{lh+h-1} \gamma^h_j(b)\gamma^j_{lh}(x) &&\text{by Prop.~\ref{equivalencias de ser un twisting map} and~\ref{escalera}}\\
& = b\gamma^h_{lh}(x) + \gamma^h_{lh+h-1}(b)\gamma^{lh+h-1}_{lh}(x) && \text{by Lemma~\ref{diagonales}}\\
& = b\gamma^h_{lh}(x) +\gamma^h_{lh+h-1}(b)\gamma^{lh}_{lh}(1). && \text{by Proposition~\ref{gamma^{lh+r}_lh(b_j)}}\\
& = b\gamma^h_{lh}(x) + \gamma^h_{lh+h-1}(b). && \text{by Prop.~\ref{equivalencias de ser un twisting map}}
\intertext{On the other hand,}
\gamma^h_{lh}(xb) & =\sum_{j=1}^{lh} \gamma^h_j(x)\gamma^j_{lh}(b) && \text{by Prop.~\ref{equivalencias de ser un twisting map} and~\ref{triangular}}\\
& =\gamma^h_{1}(x)\gamma^{1}_{lh}(b) + \gamma^h_{lh}(x)b && \text{by Lemma~\ref{diagonales}}\\
& =\gamma^h_{lh}(x)b. && \text{by Lemma~\ref{bloque}}
\end{align*}
So, $\gamma^h_{lh+h-1}(b)=0$, since $b$ is central, and then $h\gamma^1_{lh}(b)=0$, by item~(3) of Lemma~\ref{diagonales}. But this implies that $\gamma^1_{lh}(b)=0$, since $h$ is cancelable in $B$.

\smallskip

\noindent b) If $j_0=lh+j$ for some $l\ge 1$ and $1\le j<h$, then on one hand
\allowdisplaybreaks
\begin{align*}
\quad \gamma^1_{lh}(bb_j) & =\sum_{i=0}^{lh+h-1} \gamma^1_i(b)\gamma^i_{lh}(b_j) && \text{by Prop.~\ref{equivalencias de ser un twisting map} and~\ref{escalera}}\\
& = \gamma^1_1(b)\gamma^1_{lh}(b_j) + \sum_{i=j_0}^{lh+h-1} \gamma^1_i(b)\gamma^i_{lh}(b_j) && \text{by inductive hypothesis}\\
& = \gamma^1_1(b)\gamma^1_{lh}(b_j)+\gamma^1_{lh+j}(b)\gamma^{lh}_{lh}(1) && \text{by Proposition~\ref{gamma^{lh+r}_lh(b_j)}}\\
& = b\gamma^1_{lh}(b_j)+\gamma^1_{lh+j}(b), && \text{by Prop.~\ref{eq formula para gamma_k^j} and Lemma~\ref{propiedadpolinomios}}\\
\intertext{and, on the other hand,}
\gamma^1_{lh}(b_jb) & =\sum_{i=1}^{n-1} \gamma^1_i(b_j)\gamma^i_{lh}(b) && \text{by Prop.~\ref{eq formula para gamma_k^j}}\\
& =\gamma^1_{lh}(b_j)b. && \text{by Lemma~\ref{diagonales}}
\end{align*}
Hence $\gamma^1_{lh+j}(b)=0$, which concludes the inductive step and finishes the proof of the proposition.
\end{proof}

\begin{remark} Not all the non commutative truncated polynomial extensions that satisfy Conditions~A1) and~A2), also satisfy the hypothesis of Proposition~\ref{h divide a n} or Proposition~\ref{s induces the flip}. For instance, take $A:=k[x]$ and $n:=3$. The following twisting map is not of this type
\begin{align*}
&    \gamma^1_0(b_0+xb_1)=b_1x^2\quad \text{for } b_0,b_1\in k[x^2],\\
&    \gamma^1_1(x^r)=(-1)^rx^r,\\
&    \gamma^1_2(x^r)=(-1)^{r-1}x^{r-1},
\end{align*}
Note that $B=k[x^2]$, $h=2$ does not divides $n$ and $s(C\ot B) \nsubseteq B\ot C$.
\end{remark}

\begin{proposition}\label{unicidad} Let $D$ be a $k$-algebra, $h>1$ a divisor of $n$ and $g_l\colon D\to D$ {\rm(}$0\le l<n/h{\rm)}$ $k$-linear maps satisfying $g_l(1)=0$. For each $j\ge 0$, let $g_0^j$ denote the $j$-fold composition of $g_0$. Assume that $g^h_0=0$ and that there exists $x\in D$ such that $q:=g^{h-1}_0(x)$ is invertible. Suppose furthermore that $E:=\ker g_0$ is a $k$-subalgebra of $D$ and that $g_0$ is a right $E$-linear map. Then, there exists at most one twisting map $s\colon C\ot D\to D\ot C$, such that $\gamma^1_{lh}=g_l$.
\end{proposition}

\begin{proof} Since $q^{-1}\in E$, because
$$
g_0(q^{-1})q=g_0(q^{-1}q)=0,
$$
replacing $x$ by $xq^{-1}$, we can assume that $g_0^{h-1}(x)=1$, and we do it. For $0\le j <h$, let $b_j:= g_0^{h-j-1}(x)$. Note that $b_0=1$. By formula~\eqref{eq formula para gamma_k^j}, the maps $\gamma^i_j$, with $i\ge 2$, are determined by the $\gamma^1_u$'s, with $u\le j$. Moreover, by  Proposition~\ref{gamma^{lh+r}_lh(b_j)} and Proposition~\ref{equivalencias de ser un twisting map},
$$
\gamma^{lh+r}_{lh}(b_j)=\begin{cases} 1 & \text{if $r=j$}\\ 0 & \text{if $r>j$.} \end{cases}
$$
Hence, by the Product law,
$$
\gamma^1_{lh}(ab_j) = \sum_{i=0}^{n-1}\gamma^1_i(a)\gamma^i_{lh}(b_j) = \sum_{i=0}^{lh+j-1} \gamma^1_i(a) \gamma^i_{lh}(b_j) + \gamma^1_{lh+j}(a),
$$
and so, for each $j\ge 1$, the map $\gamma^1_{lh+j}$ is determined by the $\gamma^1_i$'s with $i<lh+j$.
\end{proof}

Under suitable hypothesis it is possible to say more about the maps $\gamma^i_j$. For instance we have the following result, which will not be used in the sequel.

\begin{proposition}\label{diagonal en un bloque (sobre B)} Let $l \le \lfloor \frac{n-1}{h} \rfloor$. Assume that
\begin{equation}
\gamma^{lh}_{lh}(ba)=\gamma^{lh}_{lh}(b)\gamma^{lh}_{lh}(a)\quad\text{for all $a\in A$ and $b\in B$,} \label{eqq13}
\end{equation}
and that there exist $q_l\in \gamma_0^{h-1}(A)$ such that $\gamma_{lh}^{lh}(q_l)$ is right cance\-lable in $A$. Then
$$
\gamma^{lh+i}_{lh+j}(b)=0\quad\text{for $b\in B$ and $0\le i<j<h$.}
$$
\end{proposition}

\begin{proof} Take $x_l\in A$ such that $\gamma_0^{h-1}(x_l) = q_l$ and set $b^{(l)}_j:=\gamma_0^{h-1-j}(x_l)$ for $0\le j<h$. By the Composition law and Proposition~\ref{escalera},
\begin{equation*}
\gamma^{lh+k}_{lh}=\sum_{u=0}^{lh}\gamma^{lh}_{lh-u}\xcirc\gamma^k_u = \gamma^{lh}_{lh} \xcirc\gamma^k_0\quad\text{for all $k\ge 0$.}
\end{equation*}
So,
\begin{equation}
\gamma^{lh+k}_{lh}(b^{(l)}_j) =\gamma^{lh}_{lh}\bigl(\gamma_0^{k+h-1-j}(x_l)\bigr) =\gamma^{lh}_{lh} \bigl(\gamma_0^{k-j}(q_l)\bigr) =\begin{cases} \gamma^{lh}_{lh}(q_l) &\text{if $k=j$,}\\ 0& \text{if $k>j$,}\end{cases} \label{eqq10}
\end{equation}
where the last equality for $k>j$, follows from the fact that $q_l\in B\subseteq \ker \gamma^{k-j}_0$. Moreover, by Proposition~\ref{triangular} we know that for $b\in B$,
\begin{equation}
\gamma^u_v(b)=0\quad\text{for $u>v$.}\label{eqq12}
\end{equation}
Hence,
\begin{align*}
\gamma^{lh}_{lh}(b)\gamma^{lh}_{lh}(b^{(l)}_j) &=\gamma^{lh}_{lh} \bigl(bb^{(l)}_j\bigr) && \text{by~\eqref{eqq13}}\\
& =\bigl(M(b)M(b^{(l)}_j)\bigr)_{lh,lh} &&\text{by Corollary~\ref{M es multiplicativa}}\\
& =\sum_{k=0}^{j-1}\gamma^{lh}_{lh+k}(b) \gamma^{lh+k}_{lh}(b^{(l)}_j) +\gamma^{lh}_{lh+j}(b) \gamma^{lh}_{lh}(q_l), &&\text{by~\eqref{eqq10} and~\eqref{eqq12}}
\end{align*}
for all $b\in B$. Thus
$$
\sum_{k=1}^{j-1}\gamma^{lh}_{lh+k}(b)\gamma^{lh+k}_{lh}(b^{(l)}_j) + \gamma^{lh}_{lh+j}(b) \gamma^{lh}_{lh}(q_l) = 0.
$$
Now, an easy induction on $j$ yields $\gamma^{lh}_{lh+j}(b)=0$ for $j = 1,\dots,h-1$. By Proposition~\ref{triangular} this finishes the proof.
\end{proof}

\begin{remark} By the Product law and Propositions~\ref{escalera} and~\ref{triangular},
$$
\gamma^{lh}_{lh}(ab)=\sum_{j=0}^{n-1} \gamma^{lh}_j(a) \gamma^j_{lh}(b) =\gamma^{lh}_{lh}(a)\gamma^{lh}_{lh}(b)\quad \text{for all $a\in A$, $b\in B$ and $l< \left\lfloor \frac{n}{h} \right\rfloor$.}
$$
Hence, the first hypothesis of the previous proposition is automatically fulfilled when $B$ is included in the center of $A$. In fact $\gamma^{lh}_{lh}(b)\in B$, since
$$
\gamma^1_0\bigl(\gamma^{lh}_{lh}(b)\bigr) = \gamma^{lh+1}_{lh}(b) = 0
$$
by the Composition law and Proposition~\ref{triangular}.
\end{remark}

If the hypothesis of the previous proposition are fulfilled for all $l \le \lfloor \frac{n-1}{h}\rfloor$, then from that result, Proposition~\ref{triangular} and the fact that $\gamma^0_0 = \ide$, it follows that the matrix $M(b)$ has the following shape
$$
M(b)=
\left(
\begin{array}{l@{}c@{}c@{}c@{}r}
\begin{array}{|cccc|}\hline
    b&0&\cdots &0\\
    0&b&\cdots &0\\
    \vdots&\vdots&\ddots &\vdots \\
    0&0&\cdots &b\\ \hline
\end{array}
&
\begin{array}{cccc}
    0&\cdots& \\
    \gamma_h^1(b)&\cdots &\phantom{\cdots}&\phantom{\gamma_h^1(b)} \\
    \vdots& &&\\
    &\phantom{\gamma_h^h(b)}&&
\end{array}
&&
\begin{array}{cc}
    \cdots &0\\
     &\gamma^1_{n-1}(b) \\
    &\vdots\\
    &
\end{array}
\\
\begin{array}{cccc}
    0&\hdotsfor{2} &0\\
    \phantom{b}&\phantom{0}&\phantom{\cdots} &\phantom{0}\\
    &&&
\end{array}
&
\begin{array}{|cccc|}\hline
    \gamma_h^h(b)&0&\cdots &0\\
    0&\gamma_h^h(b)&\cdots &0\\
    \vdots& \vdots&\ddots &\vdots \\
    0&0&\cdots &\gamma_h^h(b)\\ \hline
\end{array}
&&&
\\
&
\begin{array}{cccc}
    0\phantom{nm}&\cdots&\cdots &\phantom{nm}0\\
\end{array}
&
\begin{array}{|c}\hline
    \gamma_{2h}^{2h}(b)\\
    \vdots
\end{array}
&
\begin{array}{cc}
    \cdots& \phantom{\gamma^1_{n-1}(b)}\\
    \ddots &\vdots
\end{array}
\\
\begin{array}{cc}
    \vdots& \\
    0 &\cdots
\end{array}

\end{array}
\right)
$$
for each $b\in B$.

\section{Twisting maps with $\gamma^1_0\ne 0$}
Recall that $C:=k[y]/\langle y^n\rangle$ and $A$ is an arbitrary $k$-algebra. The aim of this section is to determine all the twisting maps $s\colon C\ot A\to A\ot C$ with $\gamma^1_0\ne 0$ that satisfy:

\begin{itemize}

\smallskip

\item[-] There exist $1< h\le n$ and $x\in A$ such that $\gamma^h_0=0$ and $\gamma^{h-1}_0(x)$ is invertible,

\smallskip

\item[-] $B:=\ker \gamma^1_0$ is included in the center of $A$,

\smallskip

\item[-] $s(c\ot b) = b\ot c$ for all $c\in C$ and $b\in B$,

\end{itemize}

\begin{remark} By Proposition~\ref{h divide a n} the first condition implies that $h\mid n$; by the Product law the third condition is satisfied if and only is the $\gamma^i_j$'s are $B$-linear maps; by Proposition~\ref{unicidad} the map $s$ is determined by the maps $\gamma^1_{lh}$ ($0\le l < n/h$); by Theorem~\ref{A es libre sobre B} the algebra $A$ is free over $B$ with basis $\{1=b_0,\dots,b_{n-1}\}$, where $b_j:=\gamma^{h-j-1}_0(x)$. Finally, by Proposition~\ref{s induces the flip} we know that if $h$ is cancelable in $B$, then third condition can be replaced by that requirement that $s(C\ot B)\subseteq B\ot C$.
\end{remark}

The next theorem says that the $\gamma^1_{lh}$'s can be chosen arbitrarily.

\begin{theorem}\label{clasificacion} Let $h>1$ be a divisor of $n$ and let $g_l\colon A\to A$ {\rm(}$0\le l< n/h{\rm)}$ be $k$-linear maps satisfying $g_l(1)=0$. Assume that $g^h_0=0$ and that there exists $x$ such that $q:=g^{h-1}_0(x)$ is invertible. Suppose furthermore that $B:=\ker g_0$ is a $k$-subalgebra of the center of $A$ and that the $g_l$'s are $B$-linear maps. Then there exists a unique twisting map $s\colon C\ot A\to A\ot C$ such that $\gamma^1_{lh}=g_l$. Moreover $s(c\ot b) = b\ot c\text{ for all }c\in C\text{ and }b\in B$.
\end{theorem}

\begin{proof} Since $0=g_0(q^{-1}q)=g_0(q^{-1})q=0$, we have $q^{-1}\in B$. Replacing $x$ by $q^{-1}x$, we can assume $g^{h-1}_0(x)=1$. We set $\gamma^0_j:= \delta_{0j}\ide$ and, based on the proof of Proposition~\ref{unicidad}, for increasing $l$ we define $\gamma^r_j$ for $r\ge 1$ and $lh\le j< lh+h$, as follows:

\begin{itemize}

\item[-] First $\gamma^1_{lh} :=g_l$,

\smallskip

\item[-] Then $\gamma^r_{lh}$ for $r\ge 2$, using formula~\eqref{eq formula para gamma_k^j},

\smallskip

\item[-] Then, $\gamma^1_{lh+j}$ by
$$
\qquad\gamma^1_{lh+j}(a):=\gamma^1_{lh}(a b_j)-\sum_{k=0}^{lh+j-1}\gamma^1_k(a) \gamma^k_{lh}(b_j)\quad\text{for $1\le j < h$,}
$$
where $b_j:=\gamma^{h-1-j}_0(x)$.

\smallskip

\item[-] Finally, $\gamma^r_j$ for $r\ge 2$ and $lh+1\le j < lh+h$, using formula~\eqref{eq formula para gamma_k^j}.

\end{itemize}
By construction the maps $\gamma^r_j$ are $B$-linear and $\gamma^1_j(1)=\delta_{1j}$. Hence, by Corollary~\ref{gamma como suma}, in order to prove the theorem it suffices to show that $\gamma^n_j=0$ for $j<n$ and that the maps $\gamma^1_j$'s satisfy the Product law. To carry out this task, we will need to use the Composition law (which follows immediately from the definition of the $\gamma^r_j$'s) and that $\gamma^r_j(1)=\delta_{rj}$ for all $r$ (which follows easily from the case $r = 1$, using formula~\eqref{eq formula para gamma_k^j}. Next we will check the Product law for every block of $\gamma_j^1$'s with $lh\le j< lh+h$, and that $\gamma^n_j=0$ for $j<n$, in five steps.

\begin{description}[font=\normalfont,labelindent=0.2cm,leftmargin=\parindent,style=sameline]

\item[First step] Check that $\gamma^{lh+h-1}_{lh}(x)=1$.

\item[Second step] Verify the Product law for $\gamma^1_{lh}$.

\item[Third step] Verify the Product law for $\gamma^r_j$ with $r>1$ and $j\le lh$.

\item[Fourth step] Verify the Product law for $\gamma^1_{lh+1},\dots,\gamma^1_{lh+h-1}$.

\item[Fifth step] Check that $\gamma^i_j=0$ for $j< lh+h$ and $i\ge lh+h$.

\end{description}

\medskip

\noindent For $l=0$, we have:

\smallskip

\noindent{\rm First step.}\enspace This is true by assumption.

\smallskip

\noindent{\rm Second step.}\enspace Since the maps $\gamma^r_j$ are $B$-linear and, by Lemma~\ref{basis}, $\{b_j: 0\le j<h\}$ is a $B$-basis of $A$, it is sufficient to show that
$$
\gamma^1_0\bigl(ab_j\bigr) =\sum_{k=0}^{n-1}\gamma^1_k(a)\gamma^k_0(b_j)\quad\text{for $0\le j < h$.}
$$
For $j = 0$ this follows from the fact that $\gamma^k_0(1) =\delta_{k0}$, while, for $j>0$, this follows from the definition of the $\gamma^1_j$'s and the facts that $\gamma^j_0(b_j) = 1$ and $\gamma^k_0(b_j) = 0$ for $k>j$.

\smallskip

\noindent{\rm Third step.}\enspace Assuming that the result is valid for $r$ and proceeding as when we checked item~(2)(c) in the part $(3)\Rightarrow~(2)$ of the proof of Proposition~\ref{equivalencias de ser un twisting map}, we obtain
$$
\gamma_0^{r+1}(ab)=\gamma^1_0\bigl(\gamma^r _0(ab)\bigr) =\sum_{u=0}^{n-1} \gamma^{r+1}_u(a) \gamma^u_0(b).
$$

\smallskip

\noindent{\rm Fourth step.}\enspace Let $0<j<h$. Assume that the Product law holds for $\gamma^1_i$ with $i<j$. Then,
\allowdisplaybreaks
\begin{align*}
\gamma^1_j(ab)&=\gamma^1_0(ab b_j)-\sum_{i=0}^{j-1}\gamma_i^1(ab)\gamma^i_0(b_j)\\[3pt]
&=\sum_{u=0}^{n-1}\gamma^1_u(a)\gamma^u_0(b b_j)-\sum_{i=0}^{j-1}\sum_{u=0}^{n-1} \gamma^1_u(a)\gamma_i^u(b)\gamma^i_0(b_j) &&\text{by inductive hypothesis}\\[3pt]
&=\sum_{u=0}^{n-1}\gamma^1_u(a)\Biggl(\gamma^u_0(b b_j)-\sum_{i=0}^{j-1}\gamma_i^u(b) \gamma^i_0(b_j)\Biggr)\\[3pt]
&= \sum_{u=0}^{n-1}\gamma^1_u(a)\Biggl(\sum_{i=0}^{n-1}\gamma^u_i(b)\gamma^i_0(b_j)- \sum_{i=0}^{j-1}\gamma_i^u(b) \gamma^i_0(b_j)\Biggr) &&\parbox{1.3in}{by Second and Third steps}\\[3pt]
&=\sum_{u=0}^{n-1}\gamma^1_u(a)\gamma^u_j(b),
\end{align*}
where for the last equality we use that $\gamma^j_0(b_j)=1$ and $\gamma^i_0(b_j)=0$ for $i>j$.

\smallskip

\noindent{\rm Fifth step.}\enspace This follows from Lemma~\ref{bloque}.

\smallskip

Next, assuming we have carried out the five steps until $l-1$, we execute the five steps for~$l$.

\smallskip

\noindent{\rm First step.}\enspace By the Composition law,
$$
\gamma^{lh +h-1}_{lh}(x)=\sum_{u=0}^{lh}\gamma^h_u\bigl(\gamma^{lh-1}_{lh-u}(x)\bigr)=\gamma^h_h(1)=1,
$$
since $\gamma^h_u=0$ for $u<h$, $\gamma^{lh-1}_{lh-u}=0$ for $u>h$ and $\gamma^{lh-1}_{lh-h}(x)=1$.

\smallskip

\noindent{\rm Second step.}\enspace Since the maps $\gamma^r_j$ are $B$-linear and, by Lemma~\ref{basis}, $\{b_j: 0\le r<h\}$ is a $B$-basis of $A$, it is sufficient to show that
$$
\gamma^1_{lh}\bigl(ab_j\bigr) =\sum_{k=0}^{n-1} \gamma^1_k(a)\gamma^k_{lh} (b_j)\quad\text{for $0\le j < h$.}
$$
For $j = 0$ this follows from the fact $\gamma^k_{lh}(1) =\delta_{k,lh}$. Let now $j>0$. By the Fifth step for $l-1$,
\begin{equation}
\gamma^{lh+j}_{lh}(b_j)= \sum_{i=0}^{lh}\gamma_i^{lh}\bigl(\gamma^j_{lh-i}(b_j)\bigr)= \gamma^{lh}_{lh} \bigl(\gamma^j_0(b_j) \bigr)= 1 \label{eqqq1}
\end{equation}
and
\begin{equation}
\gamma^{lh+r}_{lh}(b_j) = \sum_{i=0}^{lh} \gamma_i^{r-j}\bigl(\gamma^{lh+j}_{lh-i}(b_j)\bigr) = \gamma^{r-j}_0 \bigl( \gamma^{lh+j}_{lh}(b_j)\bigr)=0 \quad\text{for $r>j$}.\label{eqqq2}
\end{equation}
The definition of the $\gamma^1_{lh+j}$'s for $1\le j\le h-1$ yields the desired result.

\smallskip

\noindent{\rm Third step.}\enspace Assuming that the result is valid for $r$ and arguing as in the case $l=0$ we get
$$
\gamma_j^{r+1}(ab)=\sum_{k=0}^j\gamma_k^1\bigl(\gamma_{j-k}^r(ab)\bigr) = \sum_{u=0}^{n-1} \gamma_u^{r+1}(a)\gamma^u_j(b).
$$

\medskip

\noindent{\rm Fourth step.}\enspace  Assume that the Product law holds for $\gamma^1_i$ with $i<lh+j$. Then, by the definition of $\gamma^1_{lh+j}$, the inductive hypothesis and the Third step,
\begin{align*}
\gamma^1_{lh+j}(ab) & =\gamma^1_{lh}(ab b_j)-\sum_{i=0}^{lh+j-1}\gamma_i^1(ab) \gamma^i_{lh}(b_j)\\[3pt]
& =\sum_{u=0}^{n-1}\gamma^1_u(a)\gamma^u_{lh}(b b_j)- \sum_{i=0}^{lh+j-1}\,\sum_{u=0}^{n-1} \gamma^1_u(a) \gamma_i^u(b) \gamma^i_{lh}(b_j)\\[3pt]
& =\sum_{u=0}^{n-1}\gamma^1_u(a)\Biggl(\gamma^u_{lh}(b b_j)-\sum_{i=0}^{lh+j-1} \gamma_i^u(b)\gamma^i_{lh}(b_j) \Biggr)\\[3pt]
& =\sum_{u=0}^{n-1}\gamma^1_u(a)\Biggl(\sum_{i=0}^{n-1}\gamma^u_i(b) \gamma^i_{lh}(b_j)- \sum_{i=0}^{lh+j-1} \gamma_i^u(b)\gamma^i_{lh}(b_j)\Biggr)\\[3pt]
& =\sum_{u=0}^{n-1}\gamma^1_u(a)\sum_{i=lh+j}^{n-1}\gamma^u_i(b) \gamma^i_{lh}(b_j)\\[3pt]
&=\sum_{u=0}^{n-1}\gamma^1_u(a)\gamma^u_{lh+j}(b),
\end{align*}
where the last equality follows from~\eqref{eqqq1} and \eqref{eqqq2}.

\smallskip

\noindent{\rm Fifth step.}\enspace By the Composition law, for each $i\ge lh+h$ and $j<lh+h$,
$$
\gamma^i_j=\sum_{u=0}^j\gamma^h_u\xcirc\gamma^{i-h}_{j-u}=0,
$$
since $\gamma^h_u = 0$ for $u<h$ and $\gamma^{i-h}_{j-u}=0$ for $u\ge h$.
\end{proof}

\subsection{An algorithm}\label{algoritmo} Now we give an algorithm to construct non commutative truncated polynomial extensions of a $k$-algebra $A$:

\begin{enumerate}

\smallskip

\item Take a subalgebra $B$ of the center of $A$ such that $A$ is a free $B$-module with basis $\{ b_0=1,b_1,\dots,b_{h-1} \}$,

\smallskip

\item Take $C= k[y]/\langle y^n\rangle$, where $n$ is a multiple of $h$,

\smallskip

\item Finally, choose a family $g_l\colon A\to A$ ($1\le l <n/h$) of $B$-linear maps satisfying $g_l(1)=0$.

\smallskip

\end{enumerate}
Then, there is a unique twisting map $s\colon C\ot A\to A \ot C$ such that
$$
s(y\ot a)= \sum_{j=0}^{n-1} \gamma^1_j(a) \ot y^j,
$$
where
\begin{itemize}

\item[-] $\gamma^1_0\colon A \to A$ is the $B$-linear map defined by $\gamma^1_0(b_i):= b_{i-1}$ for $i\ge 1$ and $\gamma^1_0(1):=0$,

\item[-] $\gamma^1_{lh}:=g_l$ for $1\le l <n/h$,

\item[-] $\gamma^1_{lh+j}\colon A \to A$ is the $B$-linear map defined by
\begin{equation}
\qquad\qquad\quad\gamma^1_{lh+j}(a):=\gamma^1_{lh}(a b_j)-\sum_{k=0}^{lh+j-1}\gamma^1_k(a) \gamma^k_{lh}(b_j)\quad\text{for $1\le j < h$.}\label{eq8}
\end{equation}

\smallskip

\end{itemize}

\begin{remark} Since $s(c\ot b) = b\ot c$ for all $c\in C$ and $b\in B$, the algebra $B$ is included in the center of $D:=A\ot_s C$, and so, $D$ is a free $B$-algebra of dimension $hn$.
\end{remark}

\begin{remark} As was said before, all the twisting maps $s\colon C\ot A \to A \ot C$ such that

\begin{itemize}

\smallskip

\item[-] $B:=\ker \gamma^1_0$ is a subalgebra of the center of $A$,

\smallskip

\item[-] $s(c\ot b) = b\ot c$ for all $c\in C$ and $b\in B$,

\smallskip

\item[-] there exist $h\ge 2$ and $x\in A$ such that $\gamma^h_0=0$ and $\gamma^{h-1}_0(x)$ is invertible,

\end{itemize}
are of this type. In particular, for all such algebras, $h|n$ and $A$ is free over $B$.
\end{remark}

We next apply the above algorithm to construct a very specific example of truncated noncommutative polynomial extension.

\begin{example} Let $A:=k\times k$ where $k$ is a field of characteristic different from $2$ and let $B := k(1,1)$. Let $n = h = 2$, $b_0 = (1,1)$ and $b_1 = (1,-1)$. It is evident that the $B$-linear map $\gamma^1_0\colon A\to A$, determined by the conditions $\gamma^1_0(b_1) := b_0$ and  $\gamma^1_0(b_0) := 0$, is given by
$$
\gamma^1_0(\lambda_1,\lambda_2) = \left(\frac{\lambda_1-\lambda_2}{2},\frac{\lambda_1-\lambda_2}{2}\right)
$$
A direct computation applying~\eqref{eq8} gives
$$
\gamma^1_1(\lambda_1,\lambda_2) =  (\lambda_2,\lambda_1).
$$
Hence,
$$
s\bigl(y\ot (\lambda_1,\lambda_2)\bigr) = \left(\frac{\lambda_1-\lambda_2}{2},\frac{\lambda_1-\lambda_2}{2}\right)\ot 1 + (\lambda_2,\lambda_1)\ot y.
$$
Note that by Proposition~\ref{simplicidad}, the algebra $D:=A\ot_s \frac{k[y]}{\langle y^2\rangle}$ is simple. Since $(1,1)\ot y$ is nilpotent and $\dim_k(D) = 4$, necessarily $D\simeq M_2(k)$.
\end{example}

The above example is a particular case of a general result.

\begin{proposition} Let $D:=A\ot_s C$ be an algebra constructed using the algorithm introduced in Subsection~\ref{algoritmo}. Then $D$ is simple if and only if  $B$ is a field and $h =n$, where we are using the same notations as in that subsection. Moreover, in this case, $D \simeq M_n(B)$.
\end{proposition}

\begin{proof} Suppose that $D$ is simple. Since $\gamma^h_0 = 0$, it follows from Proposition~\ref{simplicidad} that $h = n$ and
$$
Ab = ABb = AbB = Ab\gamma^{n-1}_0(A) = A,
$$
for all $b\in B\setminus\{0\}$. Hence, $B$ is a field by Remark~\ref{ker gamma^1_0}. Conversely, if $B$ is a field and $h=n$, then
$$
Ab\gamma^{n-1}_0(A) = AbB = AB = A,
$$
and so, again by Proposition~\ref{simplicidad}, the algebra $D$ is simple. The last assertion follows immediately from the fact that $\dim_B (D) = n^2$ and $1\ot y$ is nilpotent of order~$n$.
\end{proof}

\section{Upper triangular twisting maps}
The aim of this section is to study twisting maps $s\colon C\otimes A\to A\otimes C$ with $\gamma_0^1=0$. Under this assumption the low dimensional Hochschild cohomology plays a prominent role. The obstructions to inductively construct twisting maps are cohomology classes. For the sake of simplicity, given a twisting map with $\gamma_0^1=0$, we set $\alpha:=\gamma_1^1$. Moreover, we let $\alpha^m$ denote the $m$-fold composition of $\alpha$ with itself. Note that formula~\eqref{eq formula para gamma_k^j} implies $\gamma^i_j=0$ for $j<i$ and $\gamma^i_i=\alpha^i$. In particular $M$ is upper triangular. Therefore, as in the introduction, we call these twisting maps and the corresponding twisted products, upper triangular. Note moreover that, by the Product law, $\alpha$ is an algebra endomorphism. Throughout this section $\Z(A)$ denotes the center of $A$ and we set $\Delta_j:=\alpha -\alpha^j$.

\smallskip

From now on we set $C_n:=k[y]/\langle y^n\rangle$, and we let ${}_{\alpha}A_{\alpha^n}$ denote the $k$-module $A$ endowed with the $A$-bimodule structure given by $a\cdot b \cdot c := \alpha(a)b\alpha^n(c)$.

\begin{theorem}\label{obstruccion en la cohomologia} Let $s_n\colon C_n\ot A \to A\ot C_n$ be an upper triangular twisting map and let $\gamma^i_j$ ($i\ge 0$ and $0\le j< n$) be the family of $k$-linear endomorphisms of $A$ associated with $s_n$. Consider the map $F\colon A\ot A\to {}_{\alpha}A_{\alpha^n}$, defined by
$$
F(a \ot b) := \sum_{i=2}^{n-1}\gamma^1_i(a)\gamma^i_n(b),
$$
where
\begin{equation}
\gamma^i_n:= \sum_{u_1,\dots,u_i\ge 1\atop u_1+\cdots+u_i=n}\gamma^1_{u_1}\xcirc\cdots\xcirc \gamma^1_{u_i} \qquad \text{for $i\ge 2$.}\label{sumamayoig1}
\end{equation}
Then, $F$ is a normalized cocycle in the canonical Hochschild cochain complex of $A$ with coefficients in ${}_{\alpha} A_{\alpha^n}$. Moreover, there exists an upper triangular twisting map
$$
s_{n+1}\colon C_{n+1}\ot A \to A \ot C_{n+1},
$$
with the same $\gamma^1_j$'s as $s_n$ for $j=0,\dots,n - 1$, if and only if $[F]=0$ in $H^2(A,{}_{\alpha} A_{\alpha^n})$. In this case $F=-\mathrm{b}^2(\gamma^1_n)$, where $\mathrm{b}^2$
is the Hochschild coboundary.

\end{theorem}

\begin{proof} It is easy to check that $\gamma^i_n(1)=0$ for $1<i<n$. Hence $F$ is normal. We next prove that it is a cocycle. In fact
\allowdisplaybreaks
\begin{align*}
\mathrm{b}^3(F)(a\ot b \ot c) = &\alpha(a)F(b\ot c) - F(ab\ot c) + F (a\ot bc) + F(a\ot b)\alpha^n(c)\\
= & \sum_{i=2}^{n-1} \gamma^1_1(a)\gamma^1_i(b)\gamma^i_n(c) - \sum_{i=2}^{n-1}\gamma^1_i(ab)\gamma^i_n(c)\\
& + \sum_{i=2}^{n-1}\gamma^1_i(a)\gamma^i_n(bc) - \sum_{i=2}^{n-1}\gamma^1_i(a)\gamma^i_n(b)\gamma^n_n(c)\\
= & \sum_{i=2}^{n-1} \gamma^1_1(a)\gamma^1_i(b)\gamma^i_n(c) -\sum_{i=2}^{n-1}\sum_{l=1}^i \gamma^1_l(a)\gamma^l_i(b) \gamma^i_n(c) \\
& + \sum_{i=2}^{n-1}\sum_{l=i}^n\gamma^1_i(a)\gamma^i_l(b)\gamma^l_n(c) - \sum_{i=2}^{n-1}\gamma^1_i(a)\gamma^i_n(b) \gamma^n_n(c)\\
= & -\sum_{i=2}^{n-1}\sum_{l=2}^i\gamma^1_l(a)\gamma^l_i(b)\gamma^i_n(c) + \sum_{i=2}^{n-1}\sum_{l=i}^{n-1} \gamma^1_i(a) \gamma^i_l(b) \gamma^l_n(c)\\
= &\, 0.
\end{align*}
Now, note that, since $\gamma^1_0=0$,
$$
\sum_{u_1,\dots,u_i\ge 0\atop u_1+\cdots+u_i=n}\gamma^1_{u_1}\xcirc\cdots\xcirc \gamma^1_{u_i},
$$
is well defined for each $i\ge 2$ (independently of the value assigned to $\gamma^1_n$), and gives $\gamma^i_n$. Hence, we can use Corollary~\ref{gamma como suma} to conclude that there exists $s_{n+1}$ satisfying the required conditions, if and only if there is a $k$-linear map $\gamma^1_n:A\to A$ fulfilling
\begin{equation}
\gamma_n^1(ab)=\gamma^1_1(a)\gamma^1_n(b)+\gamma^1_n(a)\gamma^n_n(b)+ \sum_{i=2}^{n-1} \gamma^1_i(a)\gamma^i_n(b), \label{eqqq}
\end{equation}
or, equivalently, $\mathrm{b}^2(\gamma^1_n)(a\ot b) = -F(a\ot b)$. In fact, the maps $\gamma^i_j$ ($i\ge 2$ and $j<n$) are the same as for $s_n$, and so, if $a,b\in A$ and $j<n$, then
\begin{align*}
\gamma_j^1(ab) & = \sum_{i=0}^{n-1} \gamma_i^1(a)\gamma^i_j(b)  &&\text{since $s_n$ is a twisting map} \\
& = \sum_{i=0}^n \gamma_i^1(a)\gamma^i_j(b) &&\text{because $\gamma^n_j = 0$,}
\end{align*}
while the equality
$$
\gamma_n^1(ab) = \sum_{i=0}^n \gamma_i^1(a)\gamma^i_n(b) \quad\text{for $a,b\in A$}
$$
is the same as~\eqref{eqqq}.
\end{proof}

Next, we are going to describe the first serious obstruction to construct an upper triangular twisting map, when $\alpha$ is the identity map. In the following result $\delta_1^j$ means the $j$-fold composition of $\delta_1$.

\begin{corollary} Under the assumptions of Theorem~\ref{obstruccion en la cohomologia}, if $\alpha=\ide$, then the following facts hold:

\begin{enumerate}

\smallskip

\item Any derivation $\delta_1$ of $A$ defines a twisting map $s_3\colon C_3\ot A\to A\ot C_3$, via $\gamma^1_0:=0$, $\gamma^1_1:=\ide$ and $\gamma^1_2:=\delta_1$. Moreover, all upper triangular twisting maps from $C_3\ot A$ to $A\ot C_3$, with $\gamma^1_1=\ide$, are of this type.
\smallskip

\item Any pair of derivations $\delta_1,\delta_2$ of $A$ gives a twisting map $s_4\colon C_4\ot A\to A\ot C_4$, via $\gamma^1_0:=0$, $\gamma^1_1:=\ide$, $\gamma^1_2:=\delta_1$ and $\gamma^1_3:=\delta_1^2+\delta_2$. Moreover, all upper triangular twisting maps from $C_4\ot A$ to $A\ot C_4$, with $\gamma^1_1=\ide$, are of this type.

\smallskip

\item Let $\delta_1$ and $\delta_2$ be derivatons of $A$. Consider the upper triangular twisting map $s_4\colon C_4\ot A\to A\ot C_4$, defined by $\delta_1$ and $\delta_2$. Then, there exists an upper triangular twisting map \mbox{$s_5\colon C_5\ot A\to A\ot C_5$}, with the same $\gamma^1_1$, $\gamma^1_2$ and $\gamma^1_3$ as $s_4$, if and only if $[\delta_1]\cup [\delta_2]=0$ in $H^2(A,A)$.

\end{enumerate}
\end{corollary}

\begin{proof} Let $F\colon A\ot A\to A$ be as in Theorem~\ref{obstruccion en la cohomologia}. When $n=2$, we have $F=0$. When $n=3$, a direct computation shows that $F=2\delta_1\cup\delta_1= -\mathrm{b}^2(\delta_1^2)$. Finally, when $n=4$, we have:
\begin{align*}
F(a\ot b)&= \gamma^1_2(a)\gamma^2_4(b)+\gamma^1_3(a)\gamma^3_4(b)\\
&=\delta_1(a)(3\delta_1^2(b)+2\delta_2(b))+ 3(\delta_1^2(a)+\delta_2(a))\delta_1(b)\\
&=-\mathrm{b}^2(\delta_1^3+2 \delta_1 \xcirc \delta_2)(a\ot b)+(\delta_2\cup \delta_1)(a\ot b).
\end{align*}
Items (1), (2) and (3) follow now immediately from Theorem~\ref{obstruccion en la cohomologia}.
\end{proof}

\begin{proposition}\label{unicidaddos} Let $\alpha$ be an endomorphism of $A$. Assume that there exist $b_2,\dots, b_{n-1}\in A$ such that $\Delta_{2}(b_2),\Delta_{3}(b_3),\dots,\Delta_{n-1}(b_{n-1})$ are not zero divisors and $\{b_j,\alpha(b_j):j=2,\dots, n-1\}\subseteq \Z (A)$. Then, given elements $a_2,\dots,a_{n-1}\in A$,  there is at most one upper triangular twisting map $s\colon C_n\otimes A\to A\otimes C_n$ with $\gamma_1^1=\alpha$ and $\gamma^1_j(b_j) =a_j$ for $j=2,\dots,n-1$.
\end{proposition}

\begin{proof} Assume that there is a twisting map $s$ that satisfies the hypothesis. Since $\gamma^1_0=0$, it follows easily from Corollary~\ref{gamma como suma} that, for each $j\le n$, the maps $\gamma^1_0,\dots, \gamma^1_{j-1}$ define a twisting map $s_j\colon C_j\ot A \to A\ot C_j$. It is clear that in order to complete the proof we only need to check that the uniqueness of $s_j$ implies the one of $s_{j+1}$. By Theorem~\ref{obstruccion en la cohomologia}, to carry out this task it suffices to show that if $\delta\colon A\to {}_{\alpha}A_{\alpha^j}$ is a derivation (that is, a $1$-cocycle in the cochain Hochschild complex) which vanish in $b_j$, then $\delta=0$. But, from the equalities
$$
\delta(a)\alpha^j(b_j)=\delta(ab_j)=\delta(b_ja)=\alpha(b_j)\delta(a)=\delta(a)\alpha(b_j)\quad\text{for all $a\in A$,}
$$
it follows that $\delta=0$, since $\Delta_{j}(b_j)$ is not a zero divisor of $A$.
\end{proof}

\begin{lemma}\label{existencia} Let $s_n\colon C_n\ot A \to A\ot C_n$ and $F$ be as in Theorem~\ref{obstruccion en la cohomologia}, and let $b\in \Z(A)$. Define $G_b\colon A \to A$ by $G_b(a):=F(a\ot b)-F(b\ot a)$. Then,
$$
\mathrm{b}^2(G_b)(a\ot a')=F(a\ot a')\alpha^n(b)-\alpha(b)F(a\ot a').
$$
\end{lemma}

\begin{proof} Since
\begin{align*}
b^2(G_b)(a\ot a') =& \alpha(a)F(a'\ot b)-F(aa'\ot b)+F(a\ot b)\alpha^n(a')\\
&-\alpha(a)F(b\ot a')+ F(b\ot aa')-F(b\ot a)\alpha^n(a')
\end{align*}
and $F$ is a cocycle,
\begin{align*}
0 =& \mathrm{b}^3(F)(a\ot a'\ot b)+\mathrm{b}^3(F)(b\ot a\ot a')-\mathrm{b}^3(F)(a\ot b\ot a')\\
= &\alpha(a)F(a'\ot b)-F(aa'\ot b)+F(a\ot a'b)-F(a\ot a')\alpha^n(b)\\
&+ \alpha(b)F(a\ot a')-F(ba\ot a')+F(b\ot aa')-F(b\ot a)\alpha^n(a')\\
&-\alpha(a)F(b\ot a')+F(ab\ot a')-F(a\ot ba')+F(a\ot b)\alpha^n(a')\\
=&\mathrm{b}^2(G_b)(a\ot a')+\alpha(b)F(a\ot a')-F(a\ot a')\alpha^n(b),
\end{align*}
as desired.
\end{proof}

\begin{proposition} Let $s_n\colon C_n\ot A \to A\ot C_n$ be an upper triangular twisting map and let $b\in \Z(A)$. If $\Delta_n(b)$ is invertible and $\alpha(b),\alpha^n(b)\in \Z(A)$, then, for any $a\in \Z(A)$, there is a twisting map
$$
s_{n+1}\colon C_{n+1}\ot A \to A \ot C_{n+1}
$$
with $\gamma^1_n(b)=a$ and the same $\gamma^1_j$'s as $s_n$ for $j=0,\dots,n-1$.
\end{proposition}

\begin{proof} Set $\gamma^1_n:= a\Delta_n(b)^{-1}\Delta_n + \Delta_n(b)^{-1}G_b$, where $G_b$ is the map introduced in Lemma~\ref{existencia}. Notice that
$$
\gamma^1_n(b)= a + G_b(b)=a.
$$
Using lemma~\ref{existencia} it is easy  to check that this map fulfills $\mathrm{b}^2(\gamma^1_n)=-F$, where $\mathrm{b}^2$ is the coboundary of the Hochschild complex of $A$ with coefficients in ${}_{\alpha}A_{\alpha^n}$. Then Theorem~\ref{obstruccion en la cohomologia} guarantee the existence of such $s_{n+1}$.
\end{proof}

\begin{theorem}\label{existienciaunica} Let $\alpha\colon A \to A$ be an endomorphism. Assume that there exist $b_2,\dots,b_{n-1}\in \Z(A)$ such that $\alpha(b_j),\alpha^j(b_j)\in \Z(A)$ and $\Delta_{j}(b_j)$ is invertible for all~$j$. Then, given elements $a_2,\dots,a_{n-1}\in \Z(A)$, there is a unique upper triangular twisting map
$$
s\colon C_n\otimes A\to A\otimes C_n,
$$
with ${\gamma_1^1}=\alpha$ and $\gamma^1_j(b_j)=a_j$, for $j=2,\dots,n-1$.
\end{theorem}

\begin{proof} The uniqueness it follows immediately from Proposition~\ref{unicidaddos},  and the existence can be checked easily by induction on $n$, using the previous proposition.
\end{proof}

\begin{example} It is not difficult to find examples in which the hypothesis of the previous theorem are fulfilled, for instance, we have the following cases in which $b_j=b$ for all $j$:

\begin{enumerate}

\item Let $k$ be a characteristic zero field, $A:=k[x_1,\dots,x_r]$ and take $b_j =x_1$ for all $j$ in the theorem. Let $\alpha$ be an algebra morphism with $\alpha(x_1)=x_1+\lambda$, for some $\lambda\in k\setminus\{0\}$. Then $\Delta_{j}(b)=(1-j)\lambda$ is invertible.

\item Let $K/k$ be a field extension and assume there exists an $\alpha\in Gal(K/k)$ and $b\in K$ such that $b\ne\alpha^i(b)$ for $1\le i\le n-2$. Then $\Delta_{j}(b)$ is invertible for $1<j<n$.
\end{enumerate}
\end{example}

\begin{remark} For upper triangular twisting maps, the map $\alpha := \gamma^1_1$ seems more important than in the case $\gamma^1_0\ne 0$. For instance, it is an endomorphism of algebras, moreover it is easy to check that if $\alpha$ is an injective map, then $s$ is also injective; and that if $s$ is surjective, then so is $\alpha$. Hence, when $A$ is finite dimensional, $s$ is bijective if and only if $\alpha$ is.
\end{remark}

\begin{remark}\label{twisted extensions by power series} Let $A$ be a $k$-algebra. Consider the $k$-module $A[[y]]$ consisting of the power series with coefficients in $A$. It is easy to check that having an associative and unitary algebra structure on $A[[y]]$ such that:

\begin{itemize}

\smallskip

\item[-] $A$ and $k[[y]]$ are unitary subalgebras of $A[[y]]$,

\smallskip

\item[-] $\left(\sum_{i=0}^{\infty}a_iy^i\right)y = \sum_{i=0}^{\infty}a_iy^{i+1}$,

\smallskip

\item[-] The canonical surjection $A[[y]]\to A$ is a morphism of algebras,

\smallskip

\end{itemize}
which we call an {\em upper triangular formal extension of $A$}, is the same that having a associative and unitary algebra structure on each $A[y]/\langle y^n\rangle$ such that

\begin{itemize}

\smallskip

\item[-] $A$ and $k[y]/\langle y^n\rangle$ are unitary subalgebras of $A[y]/\langle y^n\rangle$,

\smallskip

\item[-] The multiplication map takes $a\ot y^i$ to $ay^i$,

\smallskip

\item[-] The canonical surjection $A[y]/\langle y^n\rangle \to A$ is a morphism of algebras,

\smallskip

\end{itemize}
in such a way that the canonical maps
$$
\pi_n\colon \frac{A[y]}{\langle y^{n+1}\rangle}\to \frac{A[y]}{\langle y^n\rangle}
$$
are $k$-algebra morphisms. Hence, in order to construct such an algebra structure on $A[[y]]$, Theorem~\ref{obstruccion en la cohomologia} can be applied.
\end{remark}

In particular, we have the following result, which shows the close relationship between formal deformations and formal extensions (compare with~\cite[p.64]{G}).

\begin{corollary} If $H^2(A,{}_{\alpha}A_{\alpha^n}) = 0$ for all endomorphism $\alpha$ of $A$ and all $n\in \mathds{N}$, then any upper triangular truncated polynomial extension can be extended to an upper triangular formal extension.
\end{corollary}

\begin{proof} It follows immediately from Theorem~\ref{obstruccion en la cohomologia}.
\end{proof}

Consider now the following rigidity result for deformations~\cite[Corollary Sec. 3, p. 65]{G}: If $HH^2(A) = 0$, then A is rigid, i.e. any deformation is equivalent to the trivial one. In our setting we have to consider $HH^1(A)$, and we have to define the notion of equivalency.

\begin{definition} Two upper triangular formal extensions $A_s[[y]]$ and $A_t[[y]]$ are {\em equivalent}, if there is an algebra isomorphism
$$
\varphi\colon A_s[[y]]\to A_t[[y]],
$$
such that $\varphi(a) = a$ for all $a\in A$ and $\varphi(y)= y + R$, where $R\in A_t[[y]]y^2$.
\end{definition}

\begin{remark} It is easy to see that if $A_s[[y]]$ and $A_t[[y]]$ are equivalent, then the maps $\gamma^1_1$ determined by $s$ and $t$ coincide.
\end{remark}

Given an automorphism $\alpha$ of $A$ we let $A^{\alpha}[[y]]$ denote the unique upper triangular formal extension satisfying $ya = \alpha(a)y$ for all $a\in A$. We will name this extension the {\em trivial formal extension associated with $\alpha$}.

\smallskip

Let $A_s[[y]]$ be an upper triangular formal extension with $\gamma^1_1 = \alpha$ and let $(a_i)_{i>1}$ be a sequence of elements of $A$. Set $\{P^j_i : i\ge j\ge 1\}$ be the family of elements of $A$ recursively defined by
\begin{align*}
P^1_1 & := 1,\\
P^1_i & := a_i \text{ for $i>1$,} \\
P^j_i & := \sum_{u_1,\dots,u_j\ge 1 \atop u_1+\cdots+u_j=i} P^1_{u_1}\alpha^{u_1}(P^1_{u_2})\alpha^{u_1+u_2} (P^1_{u_3}) \cdots \alpha^{u_1+\cdots+u_{j-1}}(P^1_{u_j}).
\end{align*}
Note that in the definition of $P^j_i$ we only use $a_2,\dots,a_{i-j+1}$ and that $P^j_j = 1$ for all $j\ge 1$.

\medskip

Let $\varphi\colon A_s[[y]]\to A^{\alpha}[[y]]$ be the left $A$-linear map defined by
$$
\varphi(y) := y + a_2y^2 + a_3y^3 + a_4 y^4+\cdots\quad\text{and}\quad \varphi\left(\sum b_i y^i\right) := \sum b_i \varphi(y)^i,
$$

\begin{lemma}\label{lema} Let $v,u\in \mathds{N}$. The following assertions are equivalent:
\begin{enumerate}

\smallskip

\item $\varphi(y^va)-\varphi(y)^va\in A^{\alpha}[[y]]y^{v+u}$ for all $a\in A$.

\smallskip

\item We have
$$
P^v_i \alpha^i(a) = \sum_{j=v}^i \gamma^v_j(a) P^j_i\quad\text{for all $a\in A$ and $v\le i < u+v$.}
$$
\end{enumerate}
\end{lemma}

\begin{proof} Note that $\varphi(y) = P$ where $P := y + a_2y^2 + a_3y^3 + \cdots$. Clearly
$$
P^j = \sum_{i=j}^\infty P^j_i y^i.
$$
Hence, we have
\begin{equation}
\begin{aligned}
\varphi(y^v a) &= \sum_{j=v}^\infty \varphi(\gamma^v_j(a)y^j) \\
&= \sum_{j=v}^\infty \gamma^v_j(a) P^j \\
&= \sum_{j=v}^\infty \sum_{i=j}^\infty \gamma^v_j(a) P^j_i y^i\\
&= \sum_{i=v}^\infty \sum_{j=v}^i \gamma^v_j(a) P^j_i y^i.
\end{aligned}\label{eq10}
\end{equation}
On the other hand
\begin{equation}
\varphi(y)^v a = P^v a = \sum_{i=v}^\infty P^v_i y^i a = \sum_{i=v}^\infty P^v_i \alpha^i(a) y^i.\label{eq11}
\end{equation}
Comparing the coefficients of $y^i$ in~\eqref{eq10} and~\eqref{eq11} for $v\le i <v+u$, we obtain that $1)\Leftrightarrow 2)$.
\end{proof}

\begin{lemma}\label{lema2} Let $u\in \mathds{N}$. If
$$
\varphi(ya)-\varphi(y)a\in A^{\alpha}[[y]]y^{1+u}\quad\text{$\forall\, a\in A$,}
$$
then
$$
\varphi\left(y^v\sum c_j y^j\right) - \varphi(y)^v\left(\sum c_j y^j\right) \in A^{\alpha} [[y]]y^{v+u}\quad \text{$\forall\, v\in \mathds{N}$ and $\sum c_j y^j\in A_s[[y]]$.}
$$
\end{lemma}

\begin{proof} We have
\begin{align*}
\varphi\Bigl(y \sum c_j y^j\Bigr) & = \varphi\Bigl(\sum y c_j y^j\Bigr)\\
& = \sum \varphi(y c_j y^j)\\
& = \sum \varphi(y c_j) \varphi(y)^j\\
& \equiv \sum \varphi(y) c_j \varphi(y)^j\pmod{A^{\alpha}[[y]]y^{1+u}}\\
& = \varphi(y)\varphi\Bigl(\sum c_j y^j\Bigr).
\end{align*}
An inductive argument shows now that
$$
\varphi\left(y^v\sum c_j y^j\right) =\varphi(y^v)\varphi\left(\sum c_j y^j\right)\pmod{A^{\alpha}[[y]]y^{v+u}} \qquad\text{for all $v\in \mathds{N}$,}
$$
as desired.
\end{proof}

\begin{lemma}\label{lema3} Let $u,v\in \mathds{N}$ such that $u\ge v$. If
$$
P^1_i \alpha^i(a) = \sum_{j=1}^i \gamma^1_j(a) P^j_i\quad\text{for all $a\in A$ and $1\le i < u+1$.}
$$
then
$$
P^v_i \alpha^i(a) = \sum_{j=v}^i \gamma^v_j(a) P^j_i\quad\text{for all $a\in A$ and $v\le i < u+v$.}
$$
\end{lemma}

\begin{proof} Follows from Lemmas~\ref{lema} and~\ref{lema2}.
\end{proof}

\begin{proposition} The following assertions are equivalent:

\begin{enumerate}

\item The formal extension $A_s[[y]]$ is equivalent to the trivial formal extension associated with $\alpha$, via the map $\varphi$ determined by
$$
\qquad\quad\varphi(y) := y + a_2y^2 + a_3y^3 + a_4 y^4+\cdots\quad\text{and}\quad \varphi\left(\sum b_i y^i\right) := \sum b_i \varphi(y)^i.
$$

\smallskip

\item We have
\begin{equation}
P^1_i \alpha^i(a) = \sum_{j=1}^i \gamma^1_j(a) P^j_i\quad\text{for all $a\in A$ and $i \ge 1$.} \label{exttriv1}
\end{equation}

\end{enumerate}

\end{proposition}

\begin{proof} $1)\Rightarrow 2)$ by Lemma~\ref{lema}. We next prove that $2)\Rightarrow 1)$. By Lemma~\ref{lema} with $v=1$, we get that $\varphi(y)a = \varphi(ya)$ for all $a\in A$. Hence, by Lemma~\ref{lema2}
$$
\varphi\left(y^i\sum c_j y^j\right) =\varphi(y^i)\varphi\left(\sum c_j y^j\right)\quad\text{for all $i\in \mathds{N}$.}
$$
Consequently,
\begin{align*}
\varphi\left(\biggl(\sum_i b_i y^i\biggr)\biggl(\sum_j c_j y^j\biggr)\right) &= \varphi\left(\sum_i b_i \biggl(y^i \sum_j  c_j y^j\biggr)\right)\\
& = \sum_i b_i \varphi\biggl(y^i \sum_j  c_j y^j\biggr)\\
& = \sum_i b_i\varphi(y^i) \varphi\biggl(\sum_j  c_j y^j\biggr)\\
& = \varphi\biggl(\sum_i b_i y^i\biggr) \varphi\biggl(\sum_j  c_j y^j\biggr),
\end{align*}
as desired.
\end{proof}

\begin{theorem} Let $\alpha$ be an automorphism of $A$. If $H^1(A,{}_{\alpha}A_{\alpha^i})=0$ for all $i>1$, then any upper triangular formal extension with $\gamma^1_1=\alpha$ is equivalent to the trivial formal extension associated with $\alpha$.
\end{theorem}

\begin{proof} It suffices to find $(a_i)_{i>1}$ such that~\eqref{exttriv1} is fulfilled. We proceed by induction on $i$. Since $P^1_1=1$ and $\gamma_1^1=\alpha$,
$$
P^1_1 \alpha(a) =  \gamma^1_1(a) P^1_1\quad\text{for all $a\in A$.}
$$
and so condition~\eqref{exttriv1} is satisfied for $i=1$. Suppose we have $a_1,\dots,a_u$ such that~\eqref{exttriv1} holds for $i=1,\dots,u$. For $a\in A$ we define
$$
\Delta(a):= \sum_{j=2}^{u+1} \gamma^1_j(a) P^j_{u+1}.
$$
We notice that equality~\eqref{exttriv1} for $i=u+1$ is equivalent to
\begin{equation}
\Delta(a) = P^1_{u+1} \alpha^{u+1}(a) - \alpha(a) P^1_{u+1},\label{delta}
\end{equation}
and so we have to find $P^1_{u+1}$, such that the equality~\eqref{delta} is satisfied. We claim that $\Delta$ is a $(\alpha,\alpha^{u+1})$-derivation. In fact,
\begin{align*}
\Delta(ab)=&\sum_{j=2}^{u+1}\gamma^1_j(ab)P^j_{u+1}\\
=& \sum_{j=2}^{u+1}\sum_{h=1}^j \gamma^1_h(a)\gamma^h_j(b)P^j_{u+1} \\
=& \sum_{j=2}^{u+1} \gamma^1_1(a)\gamma^1_j(b)P^j_{u+1} + \sum_{j=2}^{u+1}\sum_{h=2}^j \gamma^1_h(a) \gamma^h_j(b)P^j_{u+1} \\
=& \alpha(a)\Delta(b) + \sum_{h=2}^{u+1} \gamma^1_h(a) \sum_{j=h}^{u+1} \gamma^h_j(b)P^j_{u+1}.
\end{align*}
Since, by Lemma~\ref{lema3},
$$
\sum_{j=h}^{u+1} \gamma^h_j(b) P^j_{u+1} = P^h_{u+1}\alpha^{u+1}(b),
$$
we have
$$
\Delta(ab)=\alpha(a)\Delta(b)+\sum_{h=2}^{u+1} \gamma^1_h(a) P^h_{u+1}\alpha^{u+1}(b)=\alpha(a)\Delta(b)+\Delta(a)\alpha^{u+1}(b),
$$
which proves the claim. Since $H^1(A,{}_{\alpha}A_{\alpha^{u+1}})=0$, every $(\alpha,\alpha^{u+1})$-derivation is inner, and so, there exists an element $a_{u+1}\in A$ such that
$$
\Delta(a) = a_{u+1} \alpha^{u+1}(a) - \alpha(a) a_{u+1}.
$$
Consequently~\eqref{delta} is satisfied if we take $P^1_{u+1}:=a_{u+1}$.
\end{proof}


\subsection{Non commutative truncated polynomial extensions of  $\mathbf{k}\boldsymbol{[}\mathbf{x}\boldsymbol{]}/ \boldsymbol{\langle}\mathbf{x}^{\mathbf{m}}\boldsymbol{\rangle}$}

\begin{theorem} Let $P_1,\cdots,P_{n-1}$ be polynomials in $A:=k[x]/\langle x^m\rangle$ such that $x$ divides $P_j$ for all $j$. Then there exists a unique upper triangular twisting map $s_n\colon C_n\otimes A\to A\otimes C_n$ with $\gamma^1_j(x)=P_j$ for each $j$.
\end{theorem}

\begin{proof} For $n=1$ the result is trivial (the unique twisting map is the flip). Suppose that it is true for $n=l$, and that $x$ divides $\gamma^i_j(x)$ for all $i,j<l$. Define $\gamma^2_l,\dots,\gamma^l_l$ as in the proof of Theorem~\ref{obstruccion en la cohomologia} or, which is equal, as in Proposition~\ref{variante de gamma como suma}. By that result, in order to construct $s_{l+1}$ it suffices to define $\gamma^1_l$ satisfying $\gamma^1_l(x)=P_l$ and the Product law. Consider the matrix
$$
M_{(l)}(x):=
\begin{pmatrix}
x &  0 & \cdots & 0 & 0\\
0 & \gamma^1_1(x) &\cdots & \gamma^1_{l-1}(x) & P_l\\
0 & 0 & \cdots & \gamma^2_{l-1}(x) & \gamma^2_l(x)\\
\vdots & \vdots & \ddots  & \vdots & \vdots\\
0 & 0 & \cdots & 0 & \gamma^l_l(x)
\end{pmatrix},
$$
and take $\gamma^1_l(x^h):=\bigl(M_{(l)}(x)^h\bigr)_{1l}$. By the formula of the matrix product
\begin{equation}
\gamma_l^1(x^{u+v})=\sum_{j=1}^l\gamma^1_j(x^u)\gamma^j_l(x^v), \label{3}
\end{equation}
provided that $u+v<m$. To conclude that~\eqref{3} also holds when $u+v\ge m$ it suffices to verify that $x$ divides $M_{(l)}(x)$. We leave this task to the reader.
\end{proof}

\begin{corollary} If $P_1,P_2,\cdots$ is a sequence of polynomials in $A:=k[x]/\langle x^m \rangle$ such that $x$ divides $P_j$ for all $j$, then there exists an algebra structure on $A[[y]]$ such that
\begin{itemize}

\smallskip

\item[-] $A$ and $k[[y]]$ are unitary subalgebras of $A[[y]]$,

\smallskip

\item[-] $\left(\sum_{i=0}^{\infty}a_iy^i\right)y = \sum_{i=0}^{\infty}a_iy^{i+1}$,

\smallskip

\item[-] The canonical surjection $A[[y]]\to A$ is a morphism of algebras,

\smallskip

\item[-] $yx = \sum_{i=1}^{\infty} P_i y^i$.

\smallskip

\end{itemize}

\end{corollary}

\begin{proof} Apply Remark~\ref{twisted extensions by power series}.
\end{proof}

\end{document}